\documentclass[11pt]{amsart}
\usepackage{amsfonts}
\usepackage{amsmath,amssymb,latexsym,soul,cite,mathrsfs}
\usepackage{mathtools}
\usepackage{color,enumitem,graphicx}
\usepackage[colorlinks=true,urlcolor=blue,
citecolor=red,linkcolor=blue,linktocpage,pdfpagelabels,
bookmarksnumbered,bookmarksopen]{hyperref}
\usepackage[english]{babel}

\usepackage[left=2.9cm,right=2.9cm,top=2.8cm,bottom=2.8cm]{geometry}
\usepackage[hyperpageref]{backref}

\usepackage[colorinlistoftodos]{todonotes}
\makeatletter
\providecommand\@dotsep{5}
\def\listtodoname{List of Todos}
\def\listoftodos{\@starttoc{tdo}\listtodoname}
\makeatother

\numberwithin{equation}{section}

\makeatletter
\@namedef{subjclassname@2020}{%
	\textup{2020} Mathematics Subject Classification}
\makeatother

\newtheorem{theorem}{Theorem}[section]
\newtheorem{proposition}[theorem]{Proposition}
\newtheorem{lemma}[theorem]{Lemma}
\newtheorem{corollary}[theorem]{Corollary}

\newtheorem{remark}{Remark}
\newtheorem{definition}[theorem]{Definition}

\begin{document}
	
	
	\title [A Generalized Choquard equation with weighted anisotropic Stein-Weiss...]{A Generalized Choquard equation with weighted anisotropic Stein-Weiss potential on nonreflexive Orlicz-Sobolev Spaces}

	\author{Lucas da Silva$^\ast$}
	\author{Marco Souto}

	\address[Lucas da Silva]{\newline\indent Unidade Acad\^emica de Matem\'atica
		\newline\indent{Universidade Federal de Campina Grande,}
		\newline\indent
		58429-970, Campina Grande - PB - Brazil}
	\email{\href{mailto:lucastri09@gmail.com}{ls3@academico.ufpb.br}}
	
	\address[Marco Souto]
	{\newline\indent Unidade Acad\^emica de Matem\'atica
		\newline\indent 
		Universidade Federal de Campina Grande,
		\newline\indent
		58429-970, Campina Grande - PB - Brazil}
	\email{\href{marco@dme.ufcg.edu.br}{marco@dme.ufcg.edu.br}}

	\pretolerance10000
\thanks{M. Souto was partially supported by CNPq/Brazil 309.692/2020-2}
\thanks{$^\ast$Corresponding author}	
	
	\begin{abstract}
	\noindent In this paper we investigate the existence of solution for the following nonlocal problem with anisotropic Stein–Weiss convolution term
		\begin{equation*}
	-\Delta_{\Phi}u+V(x)\phi(|u|)u=\dfrac{1}{|x|^\alpha}\left(\int_{\mathbb{R}^{N}} \dfrac{K(y)F(u(y))}{|x-y|^{\lambda}|y|^\alpha}dy\right)K(x)f(u(x)),\;\;x\in \mathbb{R}^{N} 
		\end{equation*}
	where $\alpha\geq 0$, $ N \geq 2$, $\lambda>0$ is a positive parameter, $V,K\in {C}(\mathbb R^N,[0,\infty))$ are nonnegative functions that may vanish at infinity, the function $f\in C  (\mathbb{R}, \mathbb R)$ is quasicritical and $\linebreak F(t)=\int_{0}^{t}f(s)ds$. To establish our existence and regularity results, we use the Hardy-type inequalities for Orlicz-Sobolev Space and the Stein-weiss inequality together with a variational technique based on the mountain pass theorem for a functional that is not necessarily in $C^1$. Furthermore, we also prove the existence of a ground state solution by the method of Nehari manifold in the case where the strict monotonicity condition on $f$ is not required.
	
This work incorporates the case where the $N$-function $\tilde{\Phi }$ does not verify the $\Delta_{2}$-condition.
		%
		%
		%
	\end{abstract}


	\maketitle
	
	\section{Introduction}

This paper concerns the existence of weak solution for the following nonlocal problem with
anisotropic Stein–Weiss convolution term
	\begin{equation*}
\left\{\;
\begin{aligned}
-\Delta_{\Phi}u+V(x)\phi(|u|)u&=\dfrac{1}{|x|^\alpha}\left(\int_{\mathbb{R}^{N}} \dfrac{K(y)F(u(y))}{|x-y|^{\lambda}|y|^\alpha}dy\right)K(x)f(u(x)),\;\;x\in \mathbb{R}^{N} & \\
& u\in D^{1,\Phi}(\mathbb{R}^{N})&
\end{aligned}
\right.
\eqno{(P)}
\end{equation*}
where $\alpha\geq 0$, $ N \geq 2$, $\lambda>0$ is a positive parameter, $V,K\in {C}(\mathbb R^N,[0,\infty))$ are nonnegative functions that may vanish at infinity, the function $f\in C (\mathbb{R}, \mathbb R)$ is quasicritical and $F(t)=\int_{0}^{t}f(s)ds$. It is important to recall that
\begin{equation*}
\Delta_{\Phi}u=\text{div}(\phi(|\nabla u|)\nabla u),
\end{equation*}
where $\Phi:\mathbb{R}\longrightarrow\mathbb{R}$ is a $N$-function of the form
\begin{equation}\label{*}
\Phi(t)=\int_{0}^{|t|}s\phi(s)ds,
\end{equation}  		
and $\phi:(0,\infty)\longrightarrow(0,\infty)$ is a $C^1$ function verifying some assumptions.  This type of problem  driven by a $N$-function $\Phi$ appears in a lot of physical
applications, such as Nonlinear Elasticity, Plasticity, Generalized Newtonian Fluid, Non-Newtonian Fluid and Plasma Physics. The reader can find more details about this subject in \cite{DB}, \cite{Figueiredo}, \cite{FN1} and their references.

Recently, motivated by the applications above, the quasilinear problems driven by a $N$-function $\Phi$ has been frequently studied. We would like to cite [\citenum{FN},\citenum{Bonanno1},\citen{Bonanno2},\citenum{Gossez},\citenum{Li},\citenum{Fuchs},\citenum{PH},\citenum{Cerny}, \citenum{Donald},\citenum{Le},\citenum{Mustonen}] and their references. In all of these works, the so-called $\Delta_{2}$-condition was assumed on $\Phi$ and on $\tilde{\Phi}$ (Complementary function of $\Phi$), which ensures that the Orlicz-Sobolev space $W^{1,\Phi}(\Omega ) $ and $D ^{1,\Phi} (\Omega)$ are reflexive Banach spaces. This assertion is used several times in order to get a nontrivial solution for elliptic problems taking into account the weak topology and the classical variational methods to $C^{1}$ functionals. 

In recent years, problems without the $\Delta_2$-condition of the function $ \tilde{\Phi} $ are being studied. This type of problem brings us many difficulties when we intend to apply variational methods. The first difficulty that can easily be seen is the lack of differentiability of the energy functional associated with the problem, meaning that classical minimax type results cannot be used here. To overcome this difficulty we will use some results involving the energy functional associated with the problem $(P) $ presented by Silva and Souto, see \cite{Lucas}, together with a weaker version of the mountain pass theorem for functionals that are differentiable in Gateaux . Another important difficulty in this paper is that we cannot use the standard analysis because $D^{1,\Phi}(\mathbb{ R}^N)$ might not be reflexive. This difficulty brings us many problems in order to apply variational methods. In order to overcome these difficulties, we consider the weak$^*$ topology recovering some compactness required in variational methods and one of the results involving the weak$^*$ topology is obtained in Lemma \ref{0.1} of this present paper. Some other works can be mentioned here, we refer the interested readers to \cite{Edcarlos} and\cite{AlvesandLeandro}.
%
Based on the papers above, we assume that $\phi:(0,\infty)\longrightarrow(0,\infty)$ is $C^1$ and satisfies the following hypotheses:
$$
t\longmapsto t\phi(t)\;\text{ is increasing for }\;t>0.
\leqno{(\phi_1)}
$$
$$
\displaystyle\lim_{t\rightarrow0^{+}}t\phi(t)=0 \;\;\text{ and }\;\; \displaystyle\lim_{t\rightarrow+\infty}t\phi(t)=+\infty.
\leqno{(\phi_2)}
$$$$
1\leq\ell= \inf_{t>0}{\dfrac{\phi(t)t^{2}}{\Phi(t)}}\leq \sup_{t>0}{\dfrac{\phi(t)t^{2}}{\Phi(t)}}=m<N,\;\;\;\;\ell \leq m < \ell^{*}\;\;\text{ and }\;\; m\neq1.
\leqno{(\phi_3)}
$$$$
t\longmapsto\dfrac{\phi(t)}{t^{m-2}} \text{ is noincreasing for } t > 0.
\leqno{(\phi_4)}
$$

	
When $\alpha=0$, due to the presence of the Choquard type non linearity, the problem $(P)$ is known as a Choquard equation. In that case, to show the existence of solution using variational methods, a tool of the main tool to deal with such type of equations is Hardy-Littlewood-Sobolev  inequality \cite{Loss}. Several works use this approach, we can mention [\citenum{Elieb},\citenum{Lions1},\citenum{Menzala}]

It is clear that there is a physical interpretation for Choquard type of equations, we refer to \cite{Moroz} and survey of such type of equations.

Recently, in \cite{Leandro}, Alves, Rădulescu and Tavares studied the equation $(P)$ with $V=K=1$ and $\alpha=0$ using different assumptions on the $N$-function $\Phi$. In this work, the authors aimed to show that the variational methods could be applied to establish the existence of solutions assuming that the $N$-function $\Phi$ satisfies the conditions $(\phi_{1})-(\phi_{ 3})$ with $\ell>1 $. One of the main difficulties was to prove that the energy functional associated with equation $(P)$ is differentiable. However, good conditions involving the function $f$ made it possible to show the differentiability of the energy functional and consequently allowed to guarantee the existence of a solution through the mountain pass theorem. It is also worth mentioning that in the same work, Alves, Rădulescu and Tavares extended the result to the case where $K=1$ and $V$ is one of the following potentials: periodic function, asymptotic periodic function, coercive or Bartsch-Wang-like potential.

In 2013, Alves and Souto \cite{AlvesandMarco} proved a result of existence of ground state solutions for the following Schr\"odinger equation 
\begin{equation*}
\begin {aligned}
- \Delta u + V (x)  u = K (x) f (u),  \; \mathbb{R}^{N} 
\end{aligned}
\end{equation*}
where $N \geq 3$. They assumed that $V, K : \mathbb{ R}^N \rightarrow \mathbb{ R}$ are continuous functions and satisfy the following conditions:

\noindent{\it$(K_0)$} $V>0$, $K\in L^{\infty}(\mathbb{R}^{N})$ and $K$ is positive almost everywhere.\vspace*{0.2cm}\\
\noindent{($K_1'$)} If $\{A_n\}\subset\mathbb{R}^{N}$ is a sequence of Borelian sets such that $\displaystyle\sup_{n}|A_n|< +\infty$, then
\begin{equation*}
\lim_{r\rightarrow+\infty}\int_{A_n\cap B_r^{c} (0)} K(x) dx=0,\;\text{ uniformly in }n\in\mathbb {N}.
\end{equation*}
 One of the following condition is true:
 
\noindent{($K_2'$)} $\dfrac{K}{V}\in L^{\infty}(\mathbb{ R}^N)$

or 

\noindent{($K_3'$)} $\dfrac{K(x)}{[V(x)]^{\frac{2^*-p}{2^*-2}}}\rightarrow 0$ for some $p \in (2, 2^*)$

Further, Chen and Yuan in \cite{Chen1}, considered the problem:
\begin{align*}
-\Delta u+V(x)\phi(|u|)u=\left(\int_{\mathbb{R}^{N}} \dfrac{K(y)F(u(y))}{ |x-y|^{\lambda}}dy\right)K(x)f(u(x)),\;\;x\in \mathbb{R}^{N}
\end{align*}
with similar conditions imposed by Alves and Souto in \cite{AlvesandMarco}. We emphasize that the authors substitute the conditions ($K_1 '$) and ($K_3'$) by the conditions ($K_1$) and ($K_3''$), respectively. ($K_1$) and ($K_3''$) are as follows:\vspace*{0.2cm}\\
\noindent{($K_1$)} If $\{A_n\}\subset\mathbb{R}^{N}$ is a sequence of Borelian sets such that $\displaystyle\sup_{n}|A_n|< +\infty$, then
\begin{equation*}
\lim_{r\rightarrow+\infty}\int_{A_n\cap B_r^{c} (0)} K(x)^{\frac{2N}{2N-\lambda}} dx=0,\;\text{ uniformly in }n\in\mathbb {N}.
\end{equation*}
\noindent{($K_3''$)} $\dfrac{K(x)^{\frac{2N}{2N-\lambda}}}{[V(x)]^{\frac{2^*-p}{2^*-2}}}\rightarrow 0$ for some $p \in (2, 2^*)$

The previous equation is a particular case of $(P)$ when $\alpha=0$.

A recent research has been done regarding the case where $\alpha$ is not necessarily $0$. It is worth mentioning the beautiful work of Du et al. \cite{Du}, where they investigated the following equation,
\begin{align*}
-\Delta u=\dfrac{1}{|x|^\alpha}\left(\int_{\mathbb{R}^{N}} \dfrac{|u(y)|^{2^* _{\alpha,\mu}}}{|x-y|^{\mu}|y|^\alpha}dy\right)|u(x)|^{2^* _{\alpha,\mu}-2}u,\;\;x\in \mathbb{R}^{N}
\end{align*}
where the critical exponent $2^* _{\alpha,\mu}=\frac{(2N-2\alpha-\mu)}{(N-2)}$. Other authors also presented results with anisotropic Stein–Weiss convolution term, we
refer to Alves and Shen \cite{Shen}, Zhang, Tang and
R\u{a}dulescu \cite{Zhang}, and the references therein.

A fundamental tool for studying problems with anisotropic Stein–Weiss convolution term is the Stein–Weiss inequality \cite{Stein}, that is the extension of the Hardy-Littlewood-Sobolev inequality.
\begin{proposition}\label{Hardy}[Stein–Weiss inequality] Set $t,r>1$, $\lambda\in (0,N)$ $\alpha+\beta\geq 0$ and $\alpha+\beta+\lambda\leq N$. If $1/t+1/r +(\lambda+\alpha+\beta)/N =2$ and $1-1/t-\lambda/N<\alpha/N<1-1/t$. Then there exists a constant $C_0=C(t,r,\alpha,\beta,N,\lambda)$ such that
	\begin{align}
	\left|\int_{\mathbb{ R}^{N}}\int_{\mathbb{ R}^{N}} \dfrac{g_1(x)g_2(y)}{|x|^\alpha|x-y|^{\lambda}|y|^\beta}dxdy\right|\leq C_0\lVert g_1\lVert_{ L^{t}(\mathbb{ R}^{N})}\lVert g_{2}\lVert_{ L^{t}(\mathbb{ R}^{N})}.
	\end{align}
	for all $g_1\in L^{r}(\mathbb{ R}^N)$ and $g_2\in L^{t}(\mathbb{ R}^N)$,  where $C_0$ is independent of $g_1$, $g_2$. For $\alpha = \beta = 0$, it is reduced to the Hartree type (also called the Choquard type) nonlinearity, which
	is driven by the classical Hardy-Littlewood-Sobolev inequality (See \cite{Loss}).

\end{proposition}

To the best of our knowledge, there seems to be no results on ground state solutions for the equation $(P)$ with $V$ vanishing at infinity and $\tilde{\Phi }$ verifying or not the $\Delta_{ 2}$-condition. In view of this, driven by the importance of Choquard's equations and the problems in non-reflexive spaces mentioned in the works above, in particular \cite{Lucas}, \cite{Chen1}, the first objective of this article is to guarantee the existence of a ground state solution for the equation $(P)$ and, after that, to prove regularity results. In this work we will assume that $0\leq\alpha<\lambda$ and $\lambda+2\alpha\in (0,N)\cap (0,2N- \frac{2N}{m})$. Furthermore, throughout the text we consider $s=\frac{2N}{2N-2\alpha-\lambda}$. To state the first results, we introduce the following hypotheses about the potentials $V$ and $K$:\vspace*{0.2cm}\\
\noindent{\it$(K_0)$} $V>0$, $K\in L^{\infty}(\mathbb{R}^{N})$ and $K$ is positive almost everywhere.\vspace*{0.2cm}\\
\noindent{(I)} If $\{A_n\}\subset\mathbb{R}^{N}$ is a sequence of Borelian sets such that $\displaystyle\sup_{n}|A_n|< +\infty$, then
\begin{equation*}
\lim_{r\rightarrow+\infty}\int_{A_n\cap B_r^{c} (0)} K(x)^s dx=0,\;\text{ uniformly in }n\in\mathbb {N}.
\eqno{(K_1)}
\end{equation*}
\noindent{(III)} One of the below conditions occurs:
\begin{equation*}
\dfrac{K}{V}\in L^{\infty}(\mathbb{ R}^N)
 \eqno{(K_2)}
\end{equation*}
 or there are $b_1,b_2 \in (m,\ell^*)$ and a $N$-function $B(t)=\int_{0}^{|t|}b(\tau)\tau d\tau$ verifying the following properties:
\begin{equation*}
t\longmapsto tb(t)\;\text{ is increasing for }\;t>0.
\leqno{\;\;\;\;(B_1)}
\end{equation*}
\begin{equation*}
\displaystyle\lim_{t\rightarrow0^{+}}tb(t)=0 \;\;\text{ and }\;\; \displaystyle\lim_{t\rightarrow+\infty}tb(t)=+\infty.
\leqno{\;\;\;\;(B_2)}
\end{equation*}
\begin{equation*}
	 b_1\leq \dfrac{b(t)t^2}{B(t)}\leq b_2,\; \text{ for all }\; t>0
	 \leqno{\;\;\;\;(B_3)}
\end{equation*}
\begin{equation*}
\text{ The function } B(|t|^{1/s}) \text{  is convex in }\mathbb{ R}
\leqno{\;\;\;\;(B_4)}
\end{equation*}
\noindent and 
	\begin{equation*}
\dfrac{K(x)^s}{H(x)}\longrightarrow 0 \;\text{ as } \;|x|\rightarrow+\infty
 \eqno{(K_3)}
\end{equation*}
where $\displaystyle H(x)=\min_{\tau>0} \left\{V(x)\dfrac{\Phi(\tau)}{B(\tau)}+\dfrac{\Phi_{*}(\tau)}{B(\tau)}\right\}$.

Hereafter, we say that $(V, K) \in \mathcal{K}_1$ if (I), (II) and $(K_2)$ hold. When (I), (II) and $(K_3)$ hold, then we say that $(V, K) \in \mathcal{K}_2$. These hypotheses were inspired by the paper \cite{AlvesandMarco} which deals with a local case involving the Laplacian and which were generalized in \cite{Lucas}.
In a first moment, we study the equation $(P)$ assuming that $(V, K) \in \mathcal{K}_1$. To this end, we must assume some conditions on $f$. 

We will consider $ A:\mathbb{ R} \longrightarrow [0,+\infty)$ and $ Z:\mathbb{ R} \longrightarrow [0,+\infty)$ $N$-functions given by $A(w)=\int_{0}^{|w|}ta(t )dt$ and $Z(w)=\int_{0}^{|w|}tz(t )dt$  where $ a:(0,+\infty)\longrightarrow(0,+\infty)$ and $\linebreak z:(0,+\infty)\longrightarrow(0,+\infty)$ are functions satisfying the following conditions:
\begin{equation*}
t\longmapsto ta(t)\;\text{ is increasing for }\;t>0\;\text{ and }\;t\longmapsto tz(t)\;\text{ is increasing for }\;t>0.
\leqno{(A_1)}
\end{equation*}
\begin{equation*}
\displaystyle\lim_{t\rightarrow0^{+}}ta(t)=0,\;\; \displaystyle\lim_{t\rightarrow+\infty}ta(t)=+\infty \;\text{ and }\;\displaystyle\lim_{t\rightarrow0^{+}}tz(t)=0,\;\; \displaystyle\lim_{t\rightarrow+\infty}tz(t)=+\infty.
\leqno{(A_2)}
\end{equation*}
\noindent{\it$(A_3)$} 
There exist $a_1,a_2,z_1,z_2  \in [m,\ell^*]$ with $a_1\leq a_2\leq z_1\leq z_2$ such that
\begin{equation}\label{1.1}
a_1\leq \dfrac{a(t)t^2}{A(t)}\leq a_2,\;\;\forall t>0.\vspace*{-0.2cm}
\end{equation}  
and
\begin{equation}\label{1.1'}
z_1=\inf_{t>0}\dfrac{z(t)t^2}{Z(t)} \;\;\text{ and }\;\; z_2\geq\sup_{t>0}\dfrac{z(t)t^2}{Z(t)}.\vspace*{-0.2cm}
\end{equation}
\noindent{\it$(A_4)$}  The functions $A(|t|^{1/s})$ and $Z(|t|^{1/s})$ are convex in $\mathbb{ R}$.

We assume that $f : \mathbb{ R} \rightarrow \mathbb{ R}$ is continuous and satisfies the following conditions:
%
\vspace*{0.1cm}\\
\noindent{\it$(f_1)$} 
$
\;\;\;\displaystyle \limsup_{t \to 0} \dfrac{f(t)}{\big(a(|t|)|t|^{2-s}\big)^{1/s}}=0\text{ and }\quad \displaystyle \lim_{t \to +\infty} \dfrac{f(t)}{\big(z(|t|)|t|^{2-s}\big)^{1/s}}=0.$
\vspace*{0.1cm}\\
\noindent{\it$(f_2)$} $t^{1-m/2}f(t)$ is nondecreasing on $(0,+\infty)$.\vspace*{0.1cm}\\
\noindent{\it$(f_3)$} $f(t)\geq 0$ for $t\geq 0$ and $f(t)=0$ for $t\leq 0$.\vspace*{0.1cm}\\
\noindent{\it$(f_4)$}  $\displaystyle \lim_{|t|\rightarrow\infty}\dfrac{F(t)}{|t|^{\frac{m}{2}}}=+\infty.$
\vspace*{0.1cm}\\

These conditions generalize the problem developed by Chen and Yuan in \cite{Chen1}. Assuming the conditions above, our first main result can be stated as follows.
\begin{theorem}\label{Teo1}
	Assume that $\Phi$ satisfies $(\phi_{1})- (\phi_{4})$, $0\leq\alpha<\lambda$ and $\lambda+2\alpha\in (0,N)\cap (0,2N- \frac{2N}{m})$. Suppose that $(V, K) \in \mathcal{K}_1$, $(A_1)-(A_4)$ and $(f_1)$, $(f_{2})$, $(f_{3})$, $(f_{4})$ holds. Then, problem $(P)$ possesses a nonnegative ground state solution. If $2\alpha+\lambda<2\ell$, then the nonnegative solutions are locally bounded.
\end{theorem}
	
To study the regularity of the solutions provided by Theorem \ref{Teo1}, we add the following assumptions:

\noindent{\it$(\phi_5)$} There are $0<\delta<1$, $C_1, C_2>0$ and $1<\beta\leq\ell^*$ such that 
\begin{equation*}
C_1t^{\beta -1}\leq t\phi(t)\leq C_2t^{\beta -1},\;\text{ for }\;t\in[0,\delta].
\end{equation*}
	\noindent{\it$(\phi_6)$} There are constants $\delta_0>0$ and $\delta_1>0$ such that 
\begin{equation*}
\delta_0\leq\dfrac{(\phi(t)t)'}{\phi(t)}\leq \delta_1\;\text{ for}\;t>0.
\end{equation*}

We are in position to state the following regularity result:

	\begin{theorem}\label{Teo1.1} Supoose that $\Phi$ satisfies $(\phi_{5})- (\phi_{6})$.
	Under the assumptions of Theorem \ref{Teo1}, if $K\in L^{1}(\mathbb{ R}^N)$ and $2\alpha+\lambda<2\ell$, then the
	 problem $(P)$ possesses a $C^{1,\alpha}_{loc}(\mathbb{R}^N)$ positive ground state solution.
\end{theorem}

Still regarding equation $(P)$, to study the case where $(V,K)\in \mathcal{K}_2$ we must assume that $f$ satisfies the conditions $(f_2)$, $(f_3 )$, $(f_4)$ and
   \vspace*{0.1cm}\\
   \noindent{\it$(f_5)$} 
   $
   \;\;\;\displaystyle \limsup_{t \to 0} \dfrac{f(t)}{\big(\frac{1}{s}b(|t|)|t|^{2-s}\big)^{1/s}}<\infty\quad\text{ and }\quad \displaystyle \lim_{t \to +\infty} \dfrac{f(t)}{\big(\frac{1}{s}\phi_*(|t|)|t|^{2-s}\big)^{1/s}}=0.$
      \vspace*{0.1cm}\\
Under these conditions, the next result of the existence of a nonnegative solution has the following statement:

\begin{theorem}\label{Teo2}
		Assume that $\Phi$ satisfies $(\phi_{1})- (\phi_{4})$, $0\leq\alpha<\lambda$ and $\lambda+2\alpha\in (0,N)\cap (0,2N- \frac{2N}{m})$. Suppose that $(V, K) \in \mathcal{K}_2$, $(B_1)-(B_4)$ and $(f_2)$, $(f_{3})$, $(f_{4})$, $(f_{5})$ hold. If $\Phi_{*}(|t|^{1/s})$ is convex in $\mathbb{ R}$, then the problem $(P)$ possesses a nonnegative ground state solution. 
\end{theorem}

\section{Basics On Orlicz-Sobolev Spaces}
In this section we recall some properties of Orlicz and Orlicz-Sobolev spaces, which can be found in [\citenum{Adms},\citenum{FN},\citenum{Rao},\citenum{Tienari}]. First of all, we recall that a continuous function $\Phi:\mathbb{R}\rightarrow [0,+\infty)$ is a $N$-function if:
\begin{itemize}
	\item[(i)] $\Phi$ is convex;
	\item[(ii)] $\Phi(t)=0\Leftrightarrow t=0$;
	\item[(iii)] $\Phi$ is even;
	\item[(iv)] $\displaystyle\lim_{t\rightarrow 0}\dfrac{\Phi(t)}{t}=0\text{ and }\lim_{t\rightarrow +\infty}\dfrac{\Phi(t)}{t}=+\infty$.
\end{itemize}\vspace*{0.2cm}

We say that a $N$-function $\Phi$ verifies the $\Delta_2$-condition, and we denote by $\Phi\in(\Delta_2)$, if there are constants $K>0,\; t_{0}>0$ such that $$\Phi(2t)\leq K\Phi(t),~~\forall t \geq t_0.$$ In the case of $|\Omega|=+\infty$, we will consider that $\Phi\in(\Delta_2)$ if $t_0=0$. For instance, it can be shown that $\Phi(t)=|t|^p /p$ for $p > 1$ satisfies the $\Delta_{2}$-condition, while $\Phi(t)=(e^{t^2}-1)/2$. 

If $\Omega$ is an open set of $\mathbb{R}^N$, where $N$ can be a natural number such that $N\geq 1$, and $\Phi$ a $N$-function then define the Orlicz space associated with $\Phi$ as $$L^{\Phi}(\Omega)=\left\{u\in L^1_{\text{loc}}(\Omega):~\int_\Omega \Phi\left(\frac{|u|}{\lambda}\right) dx<+\infty~\text{for some}~\lambda>0\right\}.$$ The space $L^{\Phi}(\Omega)$ is a Banach space endowed with the Luxemburg norm given by $$\|u\|_{L^{\Phi}(\Omega)}=\inf\left\{\lambda>0:\int_\Omega \Phi\left(\frac{|u|}{\lambda}\right) dx\leq 1\right\}.$$ In the case that $\Phi$ verifies $\Delta_2$-condition we have $$L^{\Phi}(\Omega)=\left\{u\in L^1_{\text{loc}}(\Omega):~\int_\Omega \Phi(|u|) dx<+\infty\right\}.$$
The complementary function $\tilde{\Phi}$ associated with $\Phi$ is given by the Legendre transformation, that is, $$\tilde{\Phi}(s)=\max_{t\geq 0}\{st-\Phi(t)\},~~\forall\; t\geq 0.$$
The functions $\Phi$ and $\tilde{\Phi}$ are complementary to each other and satisfy the inequality below $$ \tilde{\Phi}(\Phi'(t))\leq \Phi(2t),\;\;\forall\; t>0.$$ Moreover, we also have a Young type inequality given by $$st\leq \Phi(t)+\tilde{\Phi}(s),~~~\forall s,t\geq 0.$$ Using the above inequality, it is possible to establish the following Holder type inequality: $$\left|\int_{\Omega}uvdx\right|\leq 2\|u\|_{L^{\Phi}(\Omega)}\|v\|_{{L^{\tilde{\Phi}}(\Omega)}},~~\text{for all}~~u\in L^{\Phi}(\Omega)~~ \text{and}~~ v\in L^{\tilde{\Phi}}(\Omega).$$

If $|\Omega|<\infty$, the space $E^{\Phi}(\Omega)$ denotes the closing of $L^{\infty}(\Omega)$ in $L^{\Phi}( \Omega)$ with respect to the norm $\lVert \cdot\lVert_{\Phi}.$ When $|\Omega|=\infty$, the space $E^{\Phi}(\Omega)$ denotes the closure of $C^{\infty}_{0}(\Omega)$ in $L^{\Phi}(\Omega)$ with respect to norm $\lVert \cdot\lVert_{\Phi}.$ In any of these cases, $L^{\Phi}(\Omega)$ is the dual space of $E^{\tilde{\Phi}}(\Omega)$, while $L^{\tilde{\Phi}}( \Omega)$ is the dual space of $E^{{\Phi}}(\Omega)$.  Moreover, $E^{{\Phi}}(\Omega)$ and $E^{\tilde{\Phi}}(\Omega)$ are separable and all continuous functional $M:E^{{\Phi} }(\Omega)\longrightarrow\mathbb{R}$ are of the form
\begin{align*}
M(v)=\int_{\Omega} v(x)g(x)dx,\;\;\;\;\text{for some function}\;\; g\in L^{\tilde{\Phi}}(\Omega).
\end{align*}

Another important function related to function $\Phi$ it is the Sobolev conjugate function ${\Phi}_*$
of $\Phi$ defined by
\begin{align*}
\Phi_{*}^{-1}(t)=\int_{0}^{t}\dfrac{\Phi^{-1}(s)}{s^{(N+1)/N}}ds\;\;\text{for}\;\;t>0\;\;\text{ when }\;\;\int_{1}^{+\infty}\dfrac{\Phi^{-1}(s)}{s^{(N+1)/N}}ds=+\infty.
\end{align*}

\begin{lemma}\label{0.3}
	Consider $\Phi$ a $N$-function of the form \eqref{*} and satisfying $(\phi_1)$ and $(\phi_2)$. Set $$\xi_0(t)=\min\{t^{\ell},t^{m}\}~~\text{and}~~\xi_1(t)=\max\{t^{\ell},t^{m}\},~~\forall t\geq 0.$$ Then $\Phi$ satisfies $$\xi_0(t)\Phi(\rho)\leq\Phi(\rho t)\leq\xi_1(t)\Phi(\rho),~~\forall \rho,t\geq0$$ and $$
	\xi_0(\|u\|_{\Phi})\leq \int_\Omega \Phi(u)dx\leq \xi_1(\|u\|_{\Phi}), \;\;\;\forall u \in L^{\Phi}(\Omega).$$
\end{lemma}

\begin{lemma}\label{0.5}
	If $\Phi$ is an N-function of the form \eqref{*} satisfying $(\phi_1)$, $(\phi_2)$ and $(\phi_{3})$, then
	\begin{equation}\label{0.22}
	\xi_2 (t) \Phi_*(\rho)\leq \Phi_*(\rho t) \leq \xi_{3}(t)\Phi_*(\rho),\;\;\; \forall\rho, t >0
	\end{equation}
	and
	\begin{align*}
	\xi_{2}(\Arrowvert u \Arrowvert_{\Phi_*}) \leq \int_{\Omega} \Phi_*(u)dx \leq \xi_{3} (\Arrowvert u\Arrowvert_{\Phi_* }), \;\;\;\forall u \in L^{\Phi_*}(\Omega),
	\end{align*}
	where $$ \xi_2 (t) = \min \{t^{\ell^*} , t^{m^*}\} \hspace{0.1cm} \text{ and } \hspace{0.1cm} \xi_{3}(t) = \max\{t^{\ell^*}, t^{m^*}\}, \;t \geq 0.$$
\end{lemma}

\begin{lemma}\label{}
	\eqref{0.22} is equivalent to
	\begin{equation*}
	{\ell^*}\leq \dfrac{\Phi'_*(t)t^2}{\Phi_*(t)} \leq {m^*},\;\;\; \forall t >0
	\end{equation*}
\end{lemma}

\begin{lemma}\label{0.51}
	If $\Phi$ is an N-function of the form \eqref{*} satisfying $(\phi_1)$, $(\phi_2)$ and $(\phi_{3})$, then
	\begin{equation*}
	\tilde{\Phi } (\rho t) \leq t^{\frac{m}{m-1}}\tilde{\Phi }(\rho),\; \text{ for all }\;\rho>0\;\text{ and }\; 0\leq t<1.
	\end{equation*}
\end{lemma}

\begin{lemma}
	If $\Phi$ is $N$-function and $ (\int_{\Omega}\Phi (|u_n|)dx)$ is a bounded sequence, then $(u_n)$ is a bounded sequence in $L^{\Phi}(\Omega)$. When $\Phi\in(\Delta_{2})$, the equivalence is valid.
\end{lemma}


\begin{lemma}
	If $\Phi$ is an N-function of the form \eqref{*} satisfying $(\phi_1)$, $(\phi_2)$ and $(\phi_{3})$ with $\ell =1$, then $\tilde{\Phi }\notin (\Delta_{2})$.
\end{lemma}

\begin{lemma}\label{1.0}
	Let $\Omega\subset\mathbb{R}^{N}$ be a domain, if $\Phi\in (\Delta_{2})$ and $(u_n)$ is a sequence in $L^{\Phi}(\Omega)$ with $u_n\longrightarrow u$ in $L^{\Phi}(\Omega)$, there is $H\in L^{\Phi}(\Omega)$ and a subsequence $(u_{n_j})$ such that\vspace*{0.2cm}\\
	$i)\;\;\;|u_{n_j}(x)|\leq H(x)$ a.e. in $\Omega$.\vspace*{0.2cm}\\
	$ii)\;\;\;u_{n_j}(x)\longrightarrow u(x)$ a.e. in $\Omega$ and all $j\in\mathbb{N}$.
\end{lemma}

For a $N$-function $\Phi$, the corresponding Orlicz-Sobolev space is defined as the Banach space $$W^{1, \Phi}(\Omega)=\left\{ u \in L^ {\Phi}(\Omega): \dfrac{\partial u}{\partial x_{i}} \in L^{\Phi}(\Omega), i=1,...,N \right\} ,$$ with the norm
\begin{align}\label{00}
\lVert u\Arrowvert_{1, \Phi} = \Arrowvert \nabla u \Arrowvert_{\Phi} + \Arrowvert u \Arrowvert_{\Phi}.
\end{align}
%

If $\Phi\in(\Delta_{2})$, the space $D^{1,\Phi}(\mathbb{R}^{N})$ is defined to be the complement of the space $C^{ \infty}_{0}(\mathbb{R}^{N})$ with respect to the standard
\begin{align}\label{0.10}
| u|_{D^{1,\Phi}(\mathbb{R}^{N})}=\lVert u\lVert_{\Phi_{*}}+\lVert \nabla u\lVert_{\Phi}.
\end{align}
Since the Orlicz-Sobolev inequality
\begin{align}\label{0.9}
\lVert u\lVert_{\Phi_{*}}\leq S_N\lVert \nabla u\lVert_{\Phi},
\end{align}
holds for $u\in D^{1,\Phi}(\mathbb{R}^{N})$ with a constant $S_N>0$, the norm \eqref{0.10} is equivalent to the norm
\begin{align}\label{0.11}
\| u\|_{D^{1,\Phi}(\mathbb{R}^{N})}=\lVert \nabla u\lVert_{\Phi},
\end{align}
on $D^{1,\Phi}(\mathbb{R}^{N})$. In this paper, we will use \eqref{0.11} as the norm of $D^{1,\Phi}(\mathbb{R}^{N})$. Clearly $$D^{1,\Phi}(\mathbb{R}^{N})\xhookrightarrow[cont\;]{} L^{\Phi_{*}}(\mathbb{R}^{N}).$$
\vspace*{-0.1cm}
The space $L^{\Phi}(\mathbb{R}^{N})$ is separable and reflexive when the $N$-functions $\Phi$ and $\tilde{\Phi}$ satisfy the $\Delta_{2}$-condition. Knowing that $D^{1,\Phi}(\mathbb{R}^{N})$ can be seen as a closed subspace of the space $L^{\Phi_{*}}(\mathbb{R}^{ N})\times (L^{\Phi}(\mathbb{R}^{N}))^{N}$, then $D^{1,\Phi}(\mathbb{R}^{N} )$ is reflexive when the $N$-functions $\Phi$, $\tilde{\Phi}$, $\Phi_*$ and $\tilde{\Phi}_*$ satisfy the $\Delta_{2}$-condition .

The following lemma is an immediate consequence of the Banach-Alaoglu-Bourbaki theorem \cite{Brezis}.

\begin{lemma}\label{0.6}
	Assume that $\Phi$ is an N-function of the form \eqref{*} satisfying $(\phi_1)$, $(\phi_2)$ and $(\phi_{3})$. If $(u_n)\subset D^{1,\Phi}(\mathbb{R}^{N})$ is a bounded sequence, then there exists a subsequence of $(u_n)$, which we will still denote by $(u_n )$, and $u\in D^{1,\Phi}(\mathbb{R}^{N})$ such that
	\begin{align}\label{0.23}
	u_n\xrightharpoonup[\quad]{\ast} u\;\;\;\text{ in }\;L^{\Phi_{*}}(\mathbb{R}^{N})\;\;\;\;\text{ and }\;\;\;\;	\dfrac{\partial u_n}{\partial x_i}\xrightharpoonup[\quad]{\ast} \dfrac{\partial u}{\partial x_i}\;\;\;\text{ in }\;L^{\Phi}(\mathbb{R}^{N})
	\end{align}
	or equivalently,
	\begin{align*}
	\int_{\mathbb{R}^N}u_nvdx\longrightarrow \int_{\mathbb{R}^{N}}uvdx,\;\;\forall v\in E^{\tilde{\Phi}_*}(\mathbb{R}^N)\end{align*}
	and
	\begin{align*}
	\int_{\mathbb{R}^{N}}\dfrac{\partial u_n}{\partial x_i}wdx\longrightarrow\int_{\mathbb{R}^{N}}\dfrac{\partial u}{\partial x_i}wdx,\;\;\forall w \in E^{\tilde{\Phi}}(\mathbb{R}^{N}).
	\end{align*}
\end{lemma}

From now on, we denote the limit \eqref{0.23} by $u_n\xrightharpoonup[\quad]{\ast} u$ in $D^{1,\Phi}(\mathbb{R}^{N})$.
As an immediate consequence of the last lemma, we have the following corollary.

\begin{corollary}\label{0.8}
	If $(u_n)\subset D^{1,\Phi}(\mathbb{R}^{N})$ is a bounded sequence with $u_n\longrightarrow u$ in $L^{\Phi}_{loc}(\mathbb{R}^{N})$, then $ u\in D^{1,\Phi}(\mathbb{R}^{N})$.
\end{corollary}

The last lemma is crucial when the space $D^{1,\Phi}(\mathbb{R}^{N})$ can be nonreflexive. And, this happens, for example, when $\Phi_{\alpha}(t)=|t|\ln(|t|^\alpha+1)$, for  $0<\alpha<\frac{N}{N-1}-1$, because $\tilde{\Phi }_{\alpha}$ does not verify the $\linebreak\Delta_{2}$-condition. Here we emphasize that the condition $(\phi_{3})$ guarantees that $\Phi$ and $\tilde{\Phi}$ verifies the $\Delta_{2}$-condition when $\ell> 1$, for more details see Fukagai and Narukawa \cite{FN}.

\section{Preliminary results}
This section focuses on preparing some preliminaries for proving Theorems \ref{Teo1}, \ref{Teo1.1} and \ref{Teo2}. Since the potential V may vanish at infinity, we cannot study equation $(P)$ on the Sobolev space $D^{1,\Phi}(\mathbb{R }^N)$ by variational methods. As
in \cite{Lucas}, we work in the space $\linebreak E=\big\{u\in D^{1,\Phi}(\mathbb{R}^{N}): \int_{\mathbb{R}^{N}} V(x)\Phi(|u |)dx<+\infty\big\}$ with norm
\begin{align*}
\lVert u\lVert_{E}=\lVert u\lVert_{D^{1,\Phi}(\mathbb{R}^{N})} +\lVert u\lVert_{L^{\Phi}_{V}(\mathbb{R}^{N})},
\end{align*}
where $\lVert u\lVert_{L^{\Phi}_{V}(\mathbb{R}^{N})}$ is the norm of Banach space $L^{\Phi}_{V}(\mathbb{R}^{N}).$

\begin{lemma}
	$E=\overline{C^{\infty}_{0}(\mathbb{R}^{N})}^{\lVert\cdot\lVert_{E}}$ is a Banach space and 	$E$ is compactly embedded in  $L^{\Phi}_{loc}(\mathbb{R}^{N})$
\end{lemma}

\begin{lemma}\label{0.1}
	Suppose $(u_n)\subset E$ is a bounded sequence in $E$, then there is $u\in E$ such that $u_n \xrightharpoonup[\quad]{\ast} u$ in $D^{1,\Phi}(\mathbb{R}^{N})$ and  $$ \int_{\mathbb{R}^{N}} \Phi(|\nabla u|)dx\leq\displaystyle \liminf_{n\rightarrow\infty }\int_{\mathbb{R}^{N}} \Phi(|\nabla u_n|)dx.$$
\end{lemma}
\noindent {\it{Proof:}}
See Lemma $3.4$ and Lemma $4.7$ in \cite{Lucas}.

	\begin{remark}\label{5.02}
	The inequality \eqref{1.1} and \eqref{1.1'} implies the following inequalities
	$$\xi_{0,A}(t)A(\rho)\leq A(\rho t)\leq\xi_{1,A}(t)A(\rho),~~\forall \rho,t\geq0$$
	$$\xi_{0,Z}(t)Z(\rho)\leq Z(\rho t)\leq\xi_{1,Z}(t)Z(\rho),~~\forall \rho,t\geq0$$
	when
	$$\xi_{0,A}(t)=\min\{t^{a_1},t^{a_2}\}~~\text{and }~~\xi_{1,A}(t)=\max\{t^{a_1},t^{a_2}\},~~\forall t\geq 0.$$
		$$\xi_{0,Z}(t)=\min\{t^{z_1},t^{z_2}\}~~\text{and }~~\xi_{1,Z}(t)=\max\{t^{z_1},t^{z_2}\},~~\forall t\geq 0.$$ Besides by Lemma \ref{0.3} and Lemma \ref{0.5}, we have
	\begin{align*}
	\limsup_{t\rightarrow0} \dfrac{A(t)}{\Phi(t)}\leq1\;\;\text{ and }\;\;\limsup_{|t|\rightarrow\infty} \dfrac{A(t)}{\Phi_{*}(t)}\leq1
	\end{align*}
		\begin{align*}
	\limsup_{t\rightarrow0} \dfrac{Z(t)}{\Phi(t)}\leq1\;\;\text{ and }\;\;\limsup_{|t|\rightarrow\infty} \dfrac{Z(t)}{\Phi_{*}(t)}\leq1
	\end{align*}
\end{remark}
	\begin{remark}\label{5.03}
	The inequality $(B_3)$ implies the following inequalities
	$$\xi_{0,B}(t)B(\rho)\leq B(\rho t)\leq\xi_{1,B}(t)B(\rho),~~\forall \rho,t\geq0$$
	when
	$$\xi_{0,B}(t)=\min\{t^{b_1},t^{b_2}\}~~\text{and }~~\xi_{1,B}(t)=\max\{t^{b_1},t^{b_2}\},~~\forall t\geq 0.$$ Besides by Lemma \ref{0.3} and Lemma \ref{0.5}, we have
	\begin{align*}
	\limsup_{t\rightarrow0} \dfrac{B(t)}{\Phi(t)}=0\;\;\text{ and }\;\;\limsup_{|t|\rightarrow\infty} \dfrac{B(t)}{\Phi_{*}(t)}=0
	\end{align*}
\end{remark}

\begin{proposition}(Hardy-type inequality) \label{01.41}
If $(V,K)\in \mathcal{K}_1$, then the space $E$ is continuous embedded in $L^{A}_{K^s}(\mathbb{R}^{N})$ and $L^{Z}_{K^s}(\mathbb{R}^{N})$. On the other hand, if $(V,K)\in \mathcal{K}_2$, then the space $E$ is continuous embedded in $L^{B}_{K^s}(\mathbb{R}^{N})$.	
\end{proposition}
\noindent {\it{Proof:}}
The case $(V,K)\in \mathcal{K}_1$ is obvious.

Now, let us assume that $(V,K)\in \mathcal{K}_2$. As $E$ is continuously embedded in $L^{\Phi_*}(\mathbb{R}^N)$, there exists $C_1>0$ such that
	\begin{align}\label{2}
	\lVert u\lVert_{\Phi_{*}}\leq C_1\lVert u\lVert_{E},\;\;\forall u\in E.
	\end{align}
	By the condition $(K_3)$, there is $r>0$ satisfying
	\begin{align}\label{01.7}
	K(x)B(t)\leq V(x)\Phi(t)+\Phi_{*}(t),\;\;\;\;\forall t>0\;\text{ and }\;|x|\geq r.
	\end{align}
	On the other hand, by the Remark \ref{5.02}, there is a constant $C_2>0$ such that
	\begin{align*}
	B(t)\leq C_2\Phi(t)+C_2\Phi_{*}(t),\;\;\;\; \forall t>0.
	\end{align*}
	Hence, for each $x\in B_{r}(0)$,
	\begin{align}\label{01.8}
	\begin{split}
	K(x)B(t)\leq C_2\left\lVert\frac{K}{V}\right\lVert_{L^\infty( B_{r}(0))}V(x)\Phi(t)+C_2\lVert K\lVert_{\infty}\Phi_{*}(t),\;\;\;\;\forall t>0.
	\end{split}
	\end{align}
	Combining \eqref{01.7} and \eqref{01.8},
	\begin{align}\label{3}
	\begin{split}
	K(x)B(t)\leq& C_3V(x)\Phi(t)+C_3\Phi_{*}(t),\;\;\;\;\forall t>0\;\text{ and }\;x\in\mathbb{R}^N
	\end{split}
	\end{align}
	with $C_3=\max\{1,C_2\lVert K\lVert_{\infty}, C_2\left\lVert\frac{K}{V}\right\lVert_{L^\infty( B_{r}(0))}\}$. By the inequalities \eqref{2} and \eqref{2}, we get
	{\small\begin{align*}
		\int_{\mathbb{R}^{N}}K(x)B\left(\dfrac{|u|}{C_3\lVert u\lVert_{E}+C_1\lVert u\lVert_{E}}\right)dx\leq C_3	\int_{\mathbb{R}^{N}}V(x)\Phi\left(\dfrac{|u|}{\lVert u\lVert_{V,\Phi}}\right)dx+C_3	\int_{\mathbb{R}^{N}}\Phi_{*}\left(\dfrac{|u|}{\lVert u\lVert_{\Phi_{*}}}\right)dx\leq C_4
		\end{align*}}
	\noindent where $C_4$ is a positive constant that does not depend on $u$. So we can conclude that $E\subset L^{B}_{K}(\mathbb{R}^{N})$. Furthermore, there is a constant $C_5>0$ that does not depend on $u$, so $\lVert u\lVert_{K,B}\leq C_5 \lVert u\lVert_{E}$. Concluding that $E$ is continuous embedded in $L^{B}_{K^s}(\mathbb{R}^{N})$.

\qed

\begin{lemma}\label{01.290}
	Suppose that $(V,K)\in \mathcal{K}_1$ and $(f_1)$ holds. For each $u\in E$, there is a constant $C_1>0$ that does not depend on $u$, such that
	{\small\begin{align*}
	\left|	\int_{\mathbb{R}^{N}}\int_{\mathbb{R}^{N}}\dfrac{K(x)K(y)F(u(x))F(u(y))}{|x|^{\alpha}|x-y|^{\lambda}|y|^{\alpha}}dxdy\right|\leq C_1\left[\left(\int_{\mathbb{ R}^{N}}K(x)^s A(|u|) dx\right)^{\frac{2}{s}}+\left(\int_{\mathbb{ R}^{N}} K(x)^s Z(|u|) dx\right)^{\frac{2}{s}}\right].
	\end{align*}}
Furthermore, for each $u\in E$, there is a constant $C_2>0$, which does not depend on $u$, such that
	\begin{align}\label{01.133}
	\left|\int_{\mathbb{R}^{N}}\int_{\mathbb{R}^{N}}\dfrac{K(x)K(y)F(u(x))f(u(y))v(y)}{|x|^{\alpha}|x-y|^{\lambda}|y|^{\alpha}}dxdy\right|\leq C\lVert v\lVert_{E},\;\;\forall v\in E.
	\end{align}
\end{lemma}

\noindent {\it{Proof:}} By $(f_1)$, there is a constant $C>0$ such that
\begin{align}\label{01.90}
|f(t)|^s\leq  C(a(t)t^{2-s}+z(t)t^{2-s}), \;\;\;\;\forall t\in\mathbb{R}.
\end{align}
For each $t\geq0$, we have
\begin{align*}
|F(t)|\leq \int_{0}^{t}|f(\tau)|d\tau\leq \left[\int_{0}^{t}|f(\tau)|^sd \tau\right]^{\frac{1}{s}}\left[\int_{0}^{t}d\tau\right]^{\frac{s-1}{s}}\leq t ^{\frac{s-1}{s}} \left[\int_{0}^{t}|f(\tau)|^sd\tau\right]^{\frac{1}{s}} .
\end{align*}
Thus,
\begin{align}\label{1}
|F(t)|^s&\leq t^{s-1}C\left(\int_{0}^{t}\big(a(\tau)\tau^{2-s}+z(\tau)\tau^{2-s}\big)d\tau\right)\\
&\leq Ct^{s}\big(a(t)t^{2-s}+z(t)t^{2-s}\big)\quad\quad ( \tau a( \tau)\text{ and } \tau z(\tau) \text{ are increasing in } (0,\infty))\\
&\leq C\big(A(|t|)+Z(|t|)\big),\;\;\forall t\geq0
\end{align}
Similarly,
\begin{align*}
|F(t)|^s\leq C\big(A(|t|)+Z(|t|)\big),\;\;\forall t\leq0
\end{align*}
Therefore,
\begin{align*}
|F(t)|^s\leq C\big(A(|t|)+Z(|t|)\big),\;\;\forall t\in\mathbb{ R}
\end{align*}
that is,
\begin{align}\label{01.111}
\int_{\mathbb{R}^{N}} K(x)^{s}|F(u)|^{s}dx\leq C \int_{\mathbb{R}^{N}} K( x)^{s}A(|u|)dx+C  \int_{\mathbb{R}^{N}} K( x)^{s}Z(|u|)dx<\infty,\;\forall u\in E.
\end{align}
By the inequality \eqref{Hardy}, it follows that
{\small\begin{align}
\begin{split}
\left|\int_{\mathbb{R}^{N}}\int_{\mathbb{R}^{N}}\dfrac{K(x)K(y)F(u(x))F(u (y))}{|x|^{\alpha}|x-y|^{\lambda}|y|^{\alpha}}dxdy\right|\leq& \left|\int_{\mathbb{R}^{N}}K(x)^s|F(u)|^sdx\right|^{\frac{2}{s}}\\
\leq& C\left(\int_{\mathbb{R}^{N}} K( x)^{s}A(|u|)dx+\int_{\mathbb{R}^{N}} K( x)^{s}Z(|u|)dx\right)^{\frac{2}{s}}\\
\leq& C\left[\left(\int_{\mathbb{ R}^{N}}K(x)^s A(|u|) dx\right)^{\frac{2}{s}}+\left(\int_{\mathbb{R}^{N}} K( x)^{s}A(|u|)dx\right)^{\frac{2}{s}}\right],
\end{split}
\end{align}}
for every $u\in E$ and $C>0$ is a positive constant that does not depend on $u$. 

Now, consider $u,v\in E$, from \eqref{01.90} we have
\begin{align*}
\tau_{uv}=\int_{\mathbb{R}^{N}} K(x)^s|f(u)|^{s}|v|^{s}dx
\leq C \int_{\mathbb{R}^{N}} K(x)^sa(|u|)|u|^{2-s}|v|^sdx+C\int_{\mathbb{R}^{N}} K(x)^sz(|u|)|u|^{2-s}|v|^sdx.
\end{align*}
Define the functions $H:\mathbb{ R} \longrightarrow[0,\infty) $ and $P:\mathbb{ R} \longrightarrow[0,\infty) $ given by $H(t)=A(|t |^{1/s})$ and $P(t)=Z(|t|^{1/s})$. Through the assumptions imposed under $A$ and $Z$ it is possible to show that $H$ and $P$ are $N$-functions, in addition, the functions $h:(0,\infty)\longrightarrow (0,\infty)$ defined by $h(t)t=\dfrac{1}{s}a(t^{1/s})t^{(2/s) -1}$ and $p( t)t=\dfrac{1}{s}z(t^{1/s})t^{(2/s) -1}$ are increasing and satisfy
\begin{align}
H(w)=\int_{0}^{|w|}th(t)dt,\quad \text{ and }\quad	P(w)=\int_{0}^{|w|}tp(t)dt.\vspace*{-0.6cm}
\end{align}
Since $E$ is continuous embedded in $L^{A}_{K^s}(\mathbb{R}^{N})$ and $L^{Z}_{K^s}(\mathbb{R}^{N})$, we have
\begin{align*}
\int_{\mathbb{R}^{N}} K(x)^{s}\tilde{H} (a(|u|)|u|^{2-s})dx&=\int_{\mathbb{R}^{N}} K(x)^{s}\tilde{H} (|u|^sh(|u|^s))dx\\
&\leq C_1 \int_{\mathbb{R}^{N}} K(x)^{s}{H} (|u|^s)dx= C_1 \int_{\mathbb{R}^{N}} K(x)^{s}A (|u|)dx<+\infty,
\end{align*}
\begin{align*}
\int_{\mathbb{R}^{N}} K(x)^{s}\tilde{P} (z(|u|)|u|^{2-s})dx&=\int_{\mathbb{R}^{N}} K(x)^{s}\tilde{P} (|u|^sp(|u|^s))dx\\
&\leq C_1 \int_{\mathbb{R}^{N}}K(x)^{s}{P} (|u|^s)dx= C_1 \int_{\mathbb{R}^{N}} K(x)^{s}Z (|u|)dx<+\infty,
\end{align*}
\begin{align*}
\int_{\mathbb{R}^{N}} K(x)^{s}{H}  \Big(\dfrac{|v|^s}{\lVert v\lVert_{L^{A}_{K^s}(\mathbb{R}^{N})}^s}\Big)dx= \int_{\mathbb{R}^{N}} K(x)^{s}{A}  \Big(\dfrac{|v|}{\lVert v\lVert_{L^{A}_{K^s}(\mathbb{R}^{N})}} \Big)dx= 1,
\end{align*}
\begin{align*}
\int_{\mathbb{R}^{N}} K(x)^{s}{P}  \Big(\dfrac{|v|^s}{\lVert v\lVert_{L^{Z}_{K^s}(\mathbb{R}^{N})}^s}\Big)dx= \int_{\mathbb{R}^{N}} K(x)^{s}{Z}  \Big(\dfrac{|v|}{\lVert v\lVert_{L^{Z}_{K^s}(\mathbb{R}^{N})}} \Big)dx= 1,
\end{align*}
With this, we conclude that $a(|u|)|u|^{2-s}\in L^{\tilde{H}}_{K^s}(\mathbb{ R}^N)$, $ z(|u|)|u|^{2-s}\in L^{\tilde{P}}_{K^s}(\mathbb{ R}^N)$, $|v|^{s}\in L^{{H}}_{K^s}(\mathbb{ R}^N)$ and $|v|^{s}\in L^{P}_{K^s}(\mathbb{ R}^N)$ . Furthermore,
\begin{align*}
\lVert |v|^s\lVert_{L^{{H}}_{K^s}(\mathbb{ R}^N)}=\lVert v\lVert_{L^{A}_{K^s}(\mathbb{ R}^N)}^s\leq C_2\lVert v\lVert_{ E} ^s
\end{align*}
and
\begin{align*}
\lVert |v|^s\lVert_{L^{{P}}_{K^s}(\mathbb{ R}^N)}=\lVert v\lVert_{L^{Z}_{K^s}(\mathbb{ R}^N)}^s\leq C_2\lVert v\lVert_{ E} ^s
\end{align*}
where $C_i>0$, $i=1,2$, are positives constants that does not depend on $v$. By Proposition \ref{01.41}, it follows that
\begin{align}\label{0.59}
\begin{split}
\tau_{uv}\leq& C_3 \lVert a(|u|)|u|^{2-s}\lVert_{L_{K^s}^{\tilde{H}}(\mathbb{ R}^N)}\lVert|v|^ s \lVert_{L_{K^s}^{{H}}(\mathbb{ R}^N)}+C_3 \lVert Z(|u|)|u|^{2-s}\lVert_{L_{K^s}^{\tilde{P}}(\mathbb{ R}^N)}\lVert|v|^s \lVert_{L_{K^s}^{P}(\mathbb{ R}^N)}\\
\leq& C_4 (\lVert  a(|u|)|u|^{2-s}\lVert_{L_{K^s}^{\tilde{H}}(\mathbb{ R}^N)}\lVert v\lVert_{E}^s+ \lVert  z(|u|)|u|^{2-s}\lVert_{L_{K^s}^{\tilde{P}}(\mathbb{ R}^N)}\lVert v\lVert_{E}^s),
\end{split}
\end{align}
where $C_i>0$, $i=3,4$, are positives constants that does not depend on $u$ and $v$. From \eqref{0.59}, \eqref{01.111} together with the Proposition \ref{Hardy}, it follows that
{\small\begin{align*}
\Big|\int_{\mathbb{R}^{N}}\int_{\mathbb{R}^{N}}\dfrac{K(x)K(y)F(u(x))f(u(y))v(y)}{|x|^{\alpha}|x-y|^{\lambda}|y|^{\alpha}}dxdy\Big|\leq \left|\int_{\mathbb{R}^{N}}W(x)|F(u)|^sdx\right|^{\frac{1}{s}}\left|\int_{\mathbb{R}^{N}}W(x)|f(u)|^s|v|^sdx\right|^{\frac{1}{s}}\\
\leq C_5\left(\int_{\mathbb{R}^{N}} K( x)^{s}A(u)dx+\int_{\mathbb{R}^{N}}   K( x)^{s}Z(u)dx\right)^{\frac{1}{s}} \big(\lVert  a(|u|)|u|^{2-s}\lVert_{L_{K^s}^{\tilde{H}}(\mathbb{ R}^N)}+ \lVert z(|u|)|u|^{2-s}\lVert_{L_{K^s}^{\tilde{P}}(\mathbb{ R}^N)}\big)^{\frac{1}{s}}\lVert v\lVert_{E}.
\end{align*}}

\qed

\begin{lemma}\label{5.00}Assume that $(V,K)\in \mathcal{K}_1$ and $(f_1)$ holds. Let $(u_n)$ be a bounded sequence in $E$, and consider $u\in E$ such that $u_n\xrightharpoonup[\quad]{\ast} u$ in $E$. We will show the following limits
	\begin{align}\label{11.13}
	\lim_{n\rightarrow\infty}\int_{\mathbb{R}^{N}}K(x)^s|F(u_n)-F(u)|^{s}dx=0,
	\end{align}
	\begin{align}\label{11.14}
	\lim_{n\rightarrow\infty}\int_{\mathbb{R}^{N}}K(x)^s|f(u_n)u_n-f(u)u|^{s}dx=0
	\end{align}
	and
	\begin{align}\label{5.8}
	\lim_{n\rightarrow\infty}\int_{\mathbb{R}^{N}}K(x)^s|f(u_n)\varphi-f(u)\varphi|^{s}dx=0 .
	\end{align}
\end{lemma}
\begin{proof} By remark \ref{5.02}, $\displaystyle\limsup_{t \rightarrow 0}\frac{a(t)t^{2-s}}{\phi(t)t^{2-s}}\leq 1$ and $\displaystyle\limsup_{t \rightarrow \infty}\frac{z(t)t^{2-s}}{\phi_*(t)t^{2-s}}\leq 1$, then from $(f_1)$, given $\varepsilon>0$ there exist $\delta_0>0$, $\delta_1>0$ and $C_{\varepsilon}>0$ such that
	\begin{align}\label{11.011}
	|f(t)|^s\leq \varepsilon\Big(\phi(t)t^{2-s}+\phi_{*}(t)t^{2-s}\Big)+{C_{\varepsilon}}\phi_{*}(t)t^{2-s}\chi_{[\delta_0,\delta_1]}(t).
	\end{align}
	In the same way as \eqref{1}
	\begin{align}\label{11.01}
|F(t)|^s\leq \varepsilon\Big(\frac{a_2}{s}\Phi(t)+\frac{m^*}{s}\Phi_{*}(t)\Big)+\frac{C_{\varepsilon}m^*}{s}\Phi_{*}(t)\chi_{[\delta_0,\delta_1]}(t)
\end{align}
	From \eqref{11.01}, Proposition \ref{01.41} and the Sobolev inequality, it follows that the sequence ${(K(\cdot)F(u_n))}$ is bounded in $L^s(\mathbb{ R}^N)$. It is clear that $K(x)F(u_n(x)) \rightarrow K(x)F(u(x))$ a.e. in $\mathbb{ R}^N$ in the sense of subsequence. Then, by the Brézis-Lieb lemma [\citenum{Willem}, Lemma 1.32] we obtain
		\begin{align}\label{5.0}
	\int_{\mathbb{R}^{N}}K(x)^s|F(u_n)-F(u)|^{s}dx=\int_{\mathbb{R}^{N}}K(x)^s|F(u_n)|^{s}dx-\int_{\mathbb{R}^{N}}K(x)^s|F(u)|^{s}dx+o_n(1).
	\end{align}
	In view of this fact, to verify \eqref{11.13}, we only need to prove that the right side of \eqref{5.0} belongs to $o_n(1)$. 	Note that $F_n = \{x \in \mathbb{R}^N : |v_n(x)| \geq \delta_0\}$ is such that
	\begin{align*}
	\Phi_{*}(\delta_0)|F_n|\leq \int_{F_n}\Phi_{*}(|v_n(x)|)dx\leq \int_{\mathbb {R}^{N}} \Phi_{*}(|v_n(x)|)dx\leq C_1,
	\end{align*}
	for some constant $C_1>0$ that does not depend on $n$. Thus, $\displaystyle\sup_{n\in\mathbb{N}} |F_n|<+\infty$. From $(K_1)$, we have
	\begin{align*}
	\lim_{r\rightarrow+\infty}\int_{F_n\cap B_r^c (0)} K(x)^sdx=0, \;\text{ uniformly in } n\in\mathbb{N},
	\end{align*}
	thus, there is $r_0>0$, so that
	\begin{align*}
	\int_{F_n\cap B_{{r_{0}}}^c (0)} K(x)^sdx<\dfrac{\varepsilon}{\Phi_{*}(\delta_1)C_\varepsilon}, \;\;\;\forall n\in\mathbb{N}.
	\end{align*}
	Moreover, as $(v_{n})$ is bounded in $E$, there is a constant $M_1>0$ satisfying
	\begin{align*}
	\int_{\mathbb{R}^{N}}V(x)\Phi(|u_{n}|)dx\leq M_1\;\;\;\;\text{ and }\;\;\;\;
	\int_{\mathbb{R}^{N}}\Phi_{*}(|u_{n}|)dx\leq M_1,\;\;\;\;\forall n\in\mathbb{N}.
	\end{align*}
	By \eqref{11.1}, it follows that 
	\begin{align*}
	\int_{B_{r_0}^c ( 0)}K(x)^s|F(u_n)|^sdx\leq& \varepsilon C_1\left( \int_{B_{r_0}^c ( 0)}V(x)\Phi(|u_n|)dx+\int_{B_{r_0}^c ( 0)}\Phi_{*}(|u_n|)dx\right)\\
	&+ C_\varepsilon\Phi_{*}(\delta_1)\int_{F_n\cap B_{r_0}^c ( 0)}K(x)^sdx\\
	\leq &\varepsilon (C_1 M_1+1),
	\end{align*}
	for all $n\in\mathbb{N}$ where $C_1>0$ does not depend of $\varepsilon>0$. Therefore 
	\begin{align}\label{6.2}
	\limsup_{n\rightarrow+\infty} \int_{B_{r_0}^c (0)}K(x)^s|F(u_{n})|^sdx\leq \varepsilon (C_1 M_1+1).
	\end{align}
	
	On the other hand, using the compactness lemma of Strauss [\citenum{Strauss}, Theorem A.I, p. 338], it follows that
	\begin{align}\label{6.3}
	\lim_{n\rightarrow+\infty} \int_{B_{r_0} (0)}K(x)^s|F(u_{n})|^sdx=\int_{B_{r_0} (0)}K(x)^s|F(u)|^sdx.
	\end{align}
In light of this, we can conclude that
\begin{align*}
\lim_{n\rightarrow+\infty}\int_{\mathbb{R}^{N}}K(x)^s|F(u_{n})|^sdx=\int_{\mathbb{R}^{ n}}K(x)^s|F (u)|^sdx.
\end{align*}
Through this limit together with \eqref{5.0}, we will get \eqref{11.13}. Similarly, the limit \eqref{11.14} is shown. Related the limit \eqref{5.8}, it follows directly from the condition $(f_1)$ together with a version of the compactness lemma of Strauss for non-autonomous problem.
	
\end{proof}

The following result is an immediate consequence of Stein–Weiss inequality and Lemma \ref{5.00}.
\begin{lemma}\label{01.40'} Assume that $(V,K)\in \mathcal{K}_1$ and $(f_1)$ holds.
	Let $(u_n)$ be a sequence bounded in $E$ and $u\in E$ such that $u_{n}\xrightharpoonup[\quad]{\ast} u$ in $D^{1,\Phi}(\mathbb{R}^{N})$. Then
	{\footnotesize\begin{align}\label{1.15'}
		\lim_{n\rightarrow\infty} \int_{\mathbb{R}^{N}}\int_{\mathbb{R}^{N}}\dfrac{K(x)K(y)F(u_n(x))F(u_n(y))}{|x|^{\alpha}|x-y|^{\lambda}|y|^{\alpha}}dxdy= \int_{\mathbb{R}^{N}}\int_{\mathbb{R}^{N}}\dfrac{K(x)K(y)F(u(x))F(u(y))}{|x|^{\alpha}|x-y|^{\lambda}|y|^{\alpha}}dxdy,
		\end{align}}
	{\footnotesize\begin{align}\label{1.16'}
		\lim_{n\rightarrow\infty} \int_{\mathbb{R}^{N}}\int_{\mathbb{R}^{N}}\dfrac{K(x)K(y)F(u_n(x))f(u_n(y))u_n(y)}{|x|^{\alpha}|x-y|^{\lambda}|y|^{\alpha}}dxdy= \int_{\mathbb{R}^{N}}\int_{\mathbb{R}^{N}}\dfrac{K(x)K(y)F(u(x))f(u(y))u(y)}{|x|^{\alpha}|x-y|^{\lambda}|y|^{\alpha}}dxdy
		\end{align}}
	and
	{\footnotesize\begin{align}\label{5.9}
		\lim_{n\rightarrow\infty} \int_{\mathbb{R}^{N}}\int_{\mathbb{R}^{N}}\dfrac{K(x)K(y)F(u_n(x))f(u_n(y))\varphi(y)}{|x|^{\alpha}|x-y|^{\lambda}|y|^{\alpha}}dxdy= \int_{\mathbb{R}^{N}}\int_{\mathbb{R}^{N}}\dfrac{K(x)K(y)F(u(x))f(u(y))\varphi(y)}{|x|^{\alpha}|x-y|^{\lambda}|y|^{\alpha}}dxdy
		\end{align}}
	for all $\varphi\in C^{\infty}_0(\mathbb{ R}^N)$.
\end{lemma}

Now, suppose that the function $f$ satisfies the condition $(f_5)$. Given $\varepsilon>0$ there exists $\delta_0>0$, $\delta_1>0$ and $C_{\varepsilon}>0$ such that
\begin{align}\label{5.4}
|f(t)|^s\leq Cb(t)t^{2-s}+\frac{\varepsilon}{s}\phi_{*}(t)t^{2-s}+C_{\varepsilon}\phi_{*}(t)t^{2-s}\chi_{[\delta_0,\delta_1]}(t), \;\;\;\;\forall t>0
\end{align}
where $C>0$ is a constant that does not depend of $\varepsilon>0$.
Repeating the same arguments used in the proof of Lemma \ref{01.290}, we can state the following result.

\begin{lemma}\label{01.291}
	Suppose that $(V,K)\in \mathcal{K}_2$ and $(f_5)$ holds. For each $u\in E$, there is a constant $C_1>0$ that does not depend on $u$, such that
	\begin{align*}
	\left|	\int_{\mathbb{R}^{N}}\int_{\mathbb{R}^{N}}\dfrac{K(x)K(y)F(u(x))F(u(y))}{|x|^{\alpha}|x-y|^{\lambda}|y|^{\alpha}}dxdy\right|\leq C_1\left[\left(\int_{\mathbb{ R}^{N}}K(x)^s B(|u|) dx\right)^{\frac{2}{s}}+\left(\int_{\mathbb{ R}^{N}} \Phi_{*}(|u|)dx\right)^{\frac{2}{s}}\right].
	\end{align*}
	Furthermore, for each $u\in E$ there is a constant $C_2>0$, which does not depend on $u$, such that
	\begin{align}\label{5.1}
	\left|\int_{\mathbb{R}^{N}}\int_{\mathbb{R}^{N}}\dfrac{K(x)K(y)F(u(x))f(u(y))v(y)}{|x|^{\alpha}|x-y|^{\lambda}|y|^{\alpha}}dxdy\right|\leq C\lVert v\lVert_{E},\;\;\forall v\in E.
	\end{align}
\end{lemma}

\begin{lemma}
	Suppose that $(V,K)\in \mathcal{K}_2$ and $(f_5)$ holds.
	Let $(u_n)$ be a bounded sequence in $E$, and consider $u\in E$ such that $u_n\xrightharpoonup[\quad]{\ast} u$ in $E$. We will show the following limits
	\begin{align}\label{11.13'}
	\lim_{n\rightarrow\infty}\int_{\mathbb{R}^{N}}K(x)^s|F(u_n)-F(u)|^{s}dx=0,
	\end{align}
	\begin{align}\label{11.14'}
	\lim_{n\rightarrow\infty}\int_{\mathbb{R}^{N}}K(x)^s|f(u_n)u_n-f(u)u|^{s}dx=0
	\end{align}
		and
	\begin{align}\label{5.8'}
	\lim_{n\rightarrow\infty}\int_{\mathbb{R}^{N}}K(x)^s|f(u_n)\varphi-f(u)\varphi|^{s}dx=0 .
	\end{align}
\end{lemma}
\begin{proof}
	Due to similarity, it suffices to verify \eqref{11.13'}. 
	By \eqref{5.4}, for any $\varepsilon>0$ there exists $\delta_0>0$, $\delta_1>0$ and $C_{\varepsilon}>0$ such that
	\begin{align}\label{11.1'}
	|F(t)|^s\leq \varepsilon\Big(\frac{b_2}{s}B(t)+\frac{m^*}{s}\Phi_{*}(t)\Big)+\frac{C_{\varepsilon}m^*}{s}\Phi_{*}(t)\chi_{[\delta_0,\delta_1]}(t)
	\end{align}
	By the condition $(K_3)$, there is $r_0>0$ sufficiently large satisfying
	\begin{align*}
	K(x)B(t)\leq \varepsilon\left( V(x){\Phi(t)}+{\Phi_{*}(t)}\right),\;\;\;\;\forall t>0\;\text{ and }\;|x|\geq r_0.
	\end{align*}
	From the above inequalities, we have
	\begin{align}\label{44.3}
	K(x)F(t)\leq  \varepsilon C_{1}V(x)\Phi(t)+\varepsilon C_2\Phi_{*}(t)+C_\varepsilon K(x)\Phi_{*}(\delta_1)\chi_{[\delta_0,\delta_1]}(t),\;\forall t>0\text{ and }|x|\geq r_0.
	\end{align}

	From \eqref{11.1'}, Lemma \ref{01.41} and the Sobolev inequality, it follows that the sequence ${(K(\cdot)F(u_n))}$ is bounded in $L^s(\mathbb{ R}^N)$. It is clear that $K(x)F(u_n(x)) \rightarrow K(x)F(u(x))$ a.e. in $\mathbb{ R}^N$ in the sense of subsequence. Then, by the Brézis-Lieb lemma [\citen{Willem}, Lemma 1.32] we obtain
	\begin{align}\label{5.0'}
	\int_{\mathbb{R}^{N}}K(x)^s|F(u_n)-F(u)|^{s}dx=\int_{\mathbb{R}^{N}}K(x)^s|F(u_n)|^{s}dx-\int_{\mathbb{R}^{N}}K(x)^s|F(u)|^{s}dx+o_n(1)
	\end{align}
	In view of this fact, to verify \eqref{11.13'}, we only need to prove that the right side of \eqref{5.0'} belongs to $o_n(1)$. Repeating the same arguments used in the proof of Lemma \ref{5.00}, it follows that	\begin{align}\label{6.2'}
	\limsup_{n\rightarrow+\infty} \int_{B_{r_0}^c (0)}K(x)^s|F(u_{n})|^sdx\leq \varepsilon (C_1 M_1+1).
	\end{align}On the other hand, using $(f_5)$ and the compactness lemma of Strauss [\citenum{Strauss}, Theorem A.I, p. 338], it is guaranteed that
	\begin{align}\label{6.3'}
	\lim_{n\rightarrow+\infty} \int_{B_{r_0} (0)}K(x)^s|F(u_{n})|^sdx=\int_{B_{r_0} (0)}K(x)^s|F(u)|^sdx.
	\end{align}
	In light of this, we can conclude that
	\begin{align*}
	\lim_{n\rightarrow+\infty}\int_{\mathbb{R}^{N}}K(x)^s|F(u_{n})|^sdx=\int_{\mathbb{R}^{ n}}K(x)^s|F (u)|^sdx.
	\end{align*}
Through this limit together with \eqref{5.0'}, we will get \eqref{11.13'}. Similarly, we show the limit \eqref{11.14'}. Related the limit \eqref{5.8'}, it follows directly from the condition $(f_5)$ together with a version of the compactness lemma of Strauss for non-autonomous problem.
	
\end{proof}

\begin{corollary}\label{5.2}  Assume that $(V,K)\in \mathcal{K}_2$ and $(f_5)$ holds.
	Let $(u_n)$ be a sequence bounded in $E$ and $u\in E$ such that $u_{n}\xrightharpoonup[\quad]{\ast} u$ in $D^{1,\Phi}(\mathbb{R}^{N})$. Then
	{\footnotesize\begin{align}\label{5.3}
		\lim_{n\rightarrow\infty} \int_{\mathbb{R}^{N}}\int_{\mathbb{R}^{N}}\dfrac{K(x)K(y)F(u_n(x))F(u_n(y))}{|x|^{\alpha}|x-y|^{\lambda}|y|^{\alpha}}dxdy= \int_{\mathbb{R}^{N}}\int_{\mathbb{R}^{N}}\dfrac{K(x)K(y)F(u(x))F(u(y))}{|x|^{\alpha}|x-y|^{\lambda}|y|^{\alpha}}dxdy,
		\end{align}}
	{\footnotesize\begin{align}\label{11.113'}
		\lim_{n\rightarrow\infty} \int_{\mathbb{R}^{N}}\int_{\mathbb{R}^{N}}\dfrac{K(x)K(y)F(u_n(x))f(u_n(y))u_n(y)}{|x|^{\alpha}|x-y|^{\lambda}|y|^{\alpha}}dxdy= \int_{\mathbb{R}^{N}}\int_{\mathbb{R}^{N}}\dfrac{K(x)K(y)F(u(x))f(u(y))u(y)}{|x|^{\alpha}|x-y|^{\lambda}|y|^{\alpha}}dxdy
		\end{align}}
	and
{\footnotesize\begin{align}\label{5.10}
	\lim_{n\rightarrow\infty} \int_{\mathbb{R}^{N}}\int_{\mathbb{R}^{N}}\dfrac{K(x)K(y)F(u_n(x))f(u_n(y))\varphi(y)}{|x|^{\alpha}|x-y|^{\lambda}|y|^{\alpha}}dxdy= \int_{\mathbb{R}^{N}}\int_{\mathbb{R}^{N}}\dfrac{K(x)K(y)F(u(x))f(u(y))\varphi(y)}{|x|^{\alpha}|x-y|^{\lambda}|y|^{\alpha}}dxdy
	\end{align}}
for all $\varphi\in C^{\infty}_0(\mathbb{ R}^N)$.
\end{corollary}

By the inequalities \eqref{01.90} and \eqref{5.4}, together with all the results presented above, it is verified that the function
\begin{align*}
\Psi(u)=\int_{\mathbb{R}^{N}}\int_{\mathbb{R}^{N}}\dfrac{K(x)K(y)F(u_n(x))F(u_n(y))}{|x|^{\alpha}|x-y|^{\lambda}|y|^{\alpha}}dxdy,\;\;u\in E
\end{align*}
is well defined, is continuously differentiable and the Gateaux derivative $\Psi':E\longrightarrow E^*$ is given by
\begin{align*}
\Psi'(u)v=\int_{\mathbb{R}^{N}}\int_{\mathbb{R}^{N}}\dfrac{K(x)K(y)F(u_n(x))f(u_n(y))v(y)}{|x|^{\alpha}|x-y|^{\lambda}|y|^{\alpha}}dxdy,\;\;\forall u,v\in E.
\end{align*}
This fact is proved similarly to Lemma 3.2, found in \cite{Leandro}.

The necessary properties involving the functional $Q:E\rightarrow \mathbb{ R}$ given by
$$
Q(u)=\int_{\mathbb{R}^{N}} \Phi(|\nabla u|)dx+\int_{\mathbb{R}^{N}}V(x)\Phi(|u |) dx.
$$
can be found in \cite{Lucas}. In summary, the energy function associated with $(P)$ given by
{\small\begin{align*}
	J(u)=\int_{\mathbb{ R}^{N}}\Phi(|\nabla u|)dx+\int_{\mathbb{ R}^{N}}V(x)\Phi(| u|)dx-\frac{1}{2}\int_{\mathbb{R}^{N}}\int_{\mathbb{R}^{N}}\dfrac{K(x)K(y)F(u_n(x))F(u_n(y))}{|x|^{\alpha}|x-y|^{\lambda}|y|^{\alpha}}dxdy,\;\;u\in E
	\end{align*}}	
is a continuous and Gateaux-differentiable functional such that $J':E\longrightarrow E^*$ given by
{\small\begin{align*}
 J'(u)v=&\int_{\mathbb{R}^{N}} \phi(|\nabla u|)\nabla u\nabla v dx+\int_{\mathbb{R}^{N} } V(x)\phi(|u|)uv dx-\int_{\mathbb{R}^{N}}\int_{\mathbb{R}^{N}}\dfrac{K(x)K(y)F(u_n(x))f(u_n(y))v(y)}{|x|^{\alpha}|x-y|^{\lambda}|y|^{\alpha}}dxdy
\end{align*}}
is continuous from the norm topology of	$E$ to the weak$^*$-topology of $E^*$.

The convexity of the functional $Q$ together with the Gateaux-differentiability of the functional $J$ allows us to present a definition of a critical point for $J$. In this sense, we will say that $u\in E$ is a critical point for the functional $J$ if
\begin{align}\label{0.13}
Q(v)-Q(u)\geq\int_{\mathbb{R}^{N}}\int_{\mathbb{R}^{N}}\dfrac{K(x)K(y)F(u_n(x))f(u_n(y))(v(y)-u(y))}{|x|^{\alpha}|x-y|^{\lambda}|y|^{\alpha}}dxdy,\;\;\;\;\forall v\in E.
\end{align}

Our next lemma establishes that a critical point $u$ in the sense \eqref{0.13} is a weak solution for $(P)$.

\begin{proposition} \label{0.4}
	If $u\in E$ is a critical point of $J$ in $E$, then u is a weak solution to $(P)$.
\end{proposition}
\noindent {\it{Proof:}} See Lemma 4.1 in \cite{AlvesandLeandro}.

\begin{lemma}\label{5.5} 	Suppose that $(V,K)\in \mathcal{K}_1$ and $(f_1)$ hold. Then there are $\rho,\eta>0$ such that $J(u)\geq \eta$ for all $u\in E\cap \partial B_{\rho}(0)$.
\end{lemma}
\begin{proof} By Lemma \ref{01.290}, there exists a positive constant $C>0$ satisfying
{\small	\begin{align*}
	\left|	\int_{\mathbb{R}^{N}}\int_{\mathbb{R}^{N}}\dfrac{K(x)K(y)F(u(x))F(u(y))}{|x-y|^{\lambda}}dxdy\right|\leq C\left[\left(\int_{\mathbb{ R}^{N}}K(x)^s A(|u|) dx\right)^{\frac{2}{s}}+\left(\int_{\mathbb{ R}^{N}}K(x)^s Z(|u|) dx\right)^{\frac{2}{s}}\right]
	\end{align*}}
	for all $u\in E$. Hence, by defining the functional $J$ together with Proposition \ref{01.41}, we get
	\begin{align*}
	J(u)\geq& \int_{\mathbb{R}^{N}}\Phi(|\nabla u|)dx+\int_{\mathbb{R}^{N}}V(x)\Phi(|u|)dx-	C\left(\int_{\mathbb{ R}^{N}}K(x)^s A(|u|) dx\right)^{\frac{2}{s}} - C\left(\int_{\mathbb{ R}^{N}} K(x)^s Z(|u|) dx\right)^{\frac{2}{s}} \\
	\geq &\xi_0(\lVert \nabla u\lVert_{\Phi})+\xi_0(\lVert u\lVert_{V,\Phi})-	C\left(\int_{\mathbb{ R}^{N}}K(x)^s A(|u|) dx\right)^{\frac{2}{s}} - C\left(\int_{\mathbb{ R}^{N}} K(x)^s Z(|u|)  dx\right)^{\frac{2}{s}}\\
	\geq&\xi_0(\lVert \nabla u\lVert_{\Phi})+\xi_0(\lVert u\lVert_{V,\Phi})-C\big(\xi_5(\lVert u\lVert_{L_{K^s}^{A}(\mathbb{ R}^N)})\big)^{\frac{2}{s}}-C\big(\xi_{3}(\lVert u\lVert_{L_{K^s}^{Z}(\mathbb{ R}^N)})\big)^{\frac{2}{s}}\\
	\geq &\lVert \nabla u\lVert_{\Phi}^{m}+ \lVert u\lVert_{V,\Phi}^{m}-C\big(\lVert u\lVert_{L_{K^s}^{A}(\mathbb{ R}^N)}\big)^{\frac{2a_1}{s}}-C\big(\lVert u\lVert_{L_{K^s}^{Z}(\mathbb{ R}^N)}\big)^{\frac{2z_1}{s}} , 
	\end{align*}
	for $u\in E$ with $\lVert u\lVert_{E}\leq 1$ where $\displaystyle\xi_0(t)=\min_{t>0}\{t^\ell,t^m\}$, $\displaystyle\xi_3(t)=\max_{t>0}\{t^{z_1},t^{z_2}\}$ and $\displaystyle\xi_5(t)=\max_{t>0}\{t^{a_1},t^{a_2}\}$. By using the classical inequality
	\begin{align*}
	(x+y)^{\sigma}\leq 2^{\sigma -1}(x^\sigma+y^\sigma), \;\;\;x,y>0, \text{ e }\;\sigma>1,
	\end{align*}
and the continuous embedded of the space $E$ in the spaces $L_{K^s}^{A}(\mathbb{ R}^N)$ and $L_{K^s}^{Z}(\mathbb{ R}^ {N})$, we get for $u\in E$ with $\lVert u\lVert_{E}\leq 1$ that
	\begin{align*}
	J(u)\geq& C\big( \lVert \nabla u\lVert_{\Phi}+ \lVert u\lVert_{V,\Phi}\big)^{m} -C\big(\lVert u\lVert_{L_{K^s}^{A}(\mathbb{ R}^N)}\big)^{\frac{2a_1}{s}}-C\big(\lVert u\lVert_{L_{K^s}^{Z}(\mathbb{ R}^N)}\big)^{\frac{2z_1}{s}}\\
	\geq&C\lVert u\lVert_{E}^{m}- C\big(\lVert u\lVert_{E}^{\frac{2a_1}{s}}+\lVert u\lVert_{E}^{\frac{2z_1}{s}}\big),
	\end{align*}
for some constant $C>0$. As $\dfrac{2}{s}>1$, then $m<\dfrac{2a_1}{s}$ and $m<\dfrac{2z_1}{s}$. Hence, setting $\rho=\lVert u\lVert$ small enough,
	\begin{align*}
	J(u)\geq C\lVert u\lVert_{E}^{m}- C\big(\lVert u\lVert_{E}^{\frac{2a_1}{s}}+\lVert u\lVert_{E}^{\frac{2z_1}{s}}\big):=\eta>0.
	\end{align*}
	Which completes the proof.
	
\end{proof}

Using the Lemma \ref{01.291}, we can repeat the same arguments used in the proof of Lemma \ref{5.5} to prove the following lemma:

\begin{lemma}\label{5.5'} 	Suppose that $(V,K)\in \mathcal{K}_2$ and $(f_5)$ hold. Then there are $\rho,\eta>0$ such that $J(u)\geq \eta$ for all $u\in E\cap \partial B_{\rho}(0)$.
\end{lemma}

By a standard argument, the following Lemma follows from the condition $(f_4)$.
\begin{lemma}\label{4.10}
	There is $e\in E$ with $\lVert u\lVert_{E}>\rho$ and $J(e)<0$.
\end{lemma}

The previous lemmas establish the mountain pass geometry for the functinal $J$ in both cases. In what follows, let us denote by $c>0$ the mountain pass level associated with $J$, that is,
\begin{align*}
c = \inf_{\gamma \in \Gamma} \max_{t \in [0,1]} J (\gamma(t))
\end{align*}
where
\begin{align*}
\Gamma = \{\gamma \in C([0,1], X) : \; \gamma(0) = 0 \; \text{ and } \; \gamma(1) = e \}.
\end{align*}
Associated with $c$, we have a Cerami sequence $(u_n)\subset E$, that is,
\begin{align}\label{11.39}
J(u_n)\longrightarrow c\;\;\;\;\text{ and }\;\;\;\;
(1+\lVert u_n\lVert)\lVert J'(u_n)\lVert_{*}\longrightarrow 0.
\end{align}
The above sequence is obtained from the Ghoussoub-Preiss theorem, see [\citenum{Motreanu}, Theorem $5.46$].

Now, we are able to prove that the Cerami sequence given in \eqref{11.39} is bounded in $E$.

\begin{lemma}\label{5}
	Let $(u_n)$ the Cerami sequence given in \eqref{11.39}. There is a constant $M>0$ such that $J(t u_n)\leq M$ for every $t\in [0,1]$ and $n\in\mathbb{N}$.
\end{lemma}

\noindent{\bf{Proof:}} Let $t_n\in [0,1]$ be such that $J(t_n u_n)=\displaystyle\max_{t\in[0,1]} J(tu_n)$. If $t_n=0$ and $t_n=1$, we are done. Thereby, we can assume $t_n\in(0,1)$ ,and so $J'(t_n u_n)u_n=0$. From this,
\begin{align*}
mJ(t_n u_n)=&mJ(t_n u_n)-J'(t_n u_n)(t_n u_n)\\
=&\int_{\mathbb{R}^{N}}\big(m\Phi(|\nabla(t_nu_n)|)-\phi(|\nabla(t_n u_n)|)|\nabla (t_n u_n)|^{2}\big)dx\\
&+\int_{\mathbb{R}^{N}}V(x)\big(m\Phi(|t_nu_n|)-\phi(|t_n u_n|)|t_n u_n|^{2}\big)dx\\
&+\int_{\mathbb{R}^{N}}\int_{\mathbb{R}^{N}}\dfrac{K(x)K(y)\big[F(t_nu_n(x))f(t_nu_n(y))t_nu_n(y)-\dfrac{m}{2}F(t_nu_n(x))F(t_nu_n(y))\big]}{|x|^{\alpha}|x-y|^{\lambda}|y|^{\alpha}}dxdy.
\end{align*}
The conditions ${(f_{2})}$ and ${(f_{3})}$ guarantee that the functions $f(t)t-\dfrac{m}{2}F(t)$ and $F(t)$ are nondecreasing for $t>0$. The condition $(\phi_4)$ ensures that the function $ m\Phi(t)-\phi(t)t^{2}$ is increasing for $t>0$. Thus,
\begin{align*}
mJ(t_n u_n)\leq mJ(u_n)-J'(u_n)u_n=mJ(u_n)-o_n(1).
\end{align*}
Since $(J(u_n))$ is bounded, there is $M>0$ such that 
\begin{align*}
J(t u_n)\leq M,\;\;\forall t\in [0,1]\;\text{ and } \;n\in\mathbb{N}.
\end{align*}

\qed

\begin{proposition}\label{5.7}
The Cerami sequence $(u_n)$ given in \eqref{11.39} is bounded.
\end{proposition}

\noindent{\bf{Proof:}} Suppose by contradiction that $\lVert u_n\lVert_{E}\longrightarrow\infty$, then we have the following cases:\vspace*{0.1cm}\\
$i)\;\;\lVert \nabla u_n\lVert_{\Phi}\longrightarrow +\infty$ and $(\lVert u_n\lVert_{V,\Phi})$ is bounded\vspace*{0.1cm}\\
$ii)\;\;\lVert u_n\lVert_{V,\Phi}\longrightarrow \infty$ and $(\lVert \nabla u_n\lVert_{\Phi})$ is bounded\vspace*{0.1cm}\\
$iii)\;\;\lVert \nabla u_n\lVert_{\Phi}\longrightarrow +\infty$ and $\lVert u_n\lVert_{V,\Phi}\longrightarrow +\infty$.

In the case $iii)$, consider
$$w_n = \dfrac{ u_n}{\lVert u_n \lVert_{E}}, \; \; \; \; \forall n \in \mathbb{N}.$$
Since $ \lVert w_n \lVert_{E} = 1$, by Lemma \ref{0.1}, there exists $w\in E$ such that $w_n\xrightharpoonup[\quad]{\ast}w$ in $D^{1,\Phi}(\mathbb{R}^{N} )$. There are two possible cases: $w = 0$ or $w \neq0$.

\noindent{\it{Case:} $w = 0$}

Note that for every constant $\sigma>1$ there is $n_0\in \mathbb{N}$ such that $\dfrac{\sigma}{\lVert  u_n\lVert_{E}}\in[0,1]$, for $n\geq n_0.
$
Given this, we get
{\footnotesize\begin{align*}
J(t_n u_n)\geq& J( \dfrac{\sigma}{\lVert \nabla u_n\lVert_{\Phi}} u_n)\\
=&J(\sigma w_n)\\
=&\int_{\mathbb{R}^{N}}\Phi(\sigma|\nabla w_n|)dx+\int_{\mathbb{R}^{N}}V(x)\Phi(\sigma|w_n|)dx-\dfrac{1}{2}\int_{\mathbb{R}^{N}}\int_{\mathbb{R}^{N}}\dfrac{K(x)K(y)F(\sigma w_n(x))F(\sigma w_n(y))}{|x|^{\alpha}|x-y|^{\lambda}|y|^{\alpha}}dxdy\\
\geq& \sigma Q(w_n)-\dfrac{1}{2}\int_{\mathbb{R}^{N}}\int_{\mathbb{R}^{N}}\dfrac{K(x)K(y)F(\sigma w_n(x))F(\sigma w_n(y))}{|x|^{\alpha}|x-y|^{\lambda}|y|^{\alpha}}dxdy
\end{align*}}
By definition of the sequence $(w_n)$, we have
$
\lVert \nabla w_n\lVert_{\Phi}\leq 1$ and $
\lVert w_n\lVert_{V,\Phi}\leq 1,$ for all $n\in\mathbb{N}.
$
Then,
\begin{align*}
\int_{\mathbb{R}^{N}}\Phi(|\nabla w_n|)dx\geq \lVert \nabla w_n\lVert_{\Phi}^{m}\;\text{ and }\;
\int_{\mathbb{R}^{N}}V(x)\Phi(|w_n|)dx\geq \lVert w_n\lVert_{V,\Phi}^{m}.
\end{align*}
So there is $C>0$ such that
\begin{align*}
Q(w_n)\geq \lVert \nabla w_n\lVert_{\Phi}^{m}+\lVert w_n\lVert_{V,\Phi}^{m}\geq C(\lVert \nabla w_n\lVert_{\Phi}+\lVert w_n\lVert_{V,\Phi})^{m},\;\;\;\;\forall\;n\in\mathbb{N}.
\end{align*}
Thus
{\footnotesize\begin{align*}
	J(t_n u_n)
	\geq& \sigma C (\lVert w_n\lVert_{ E})^m-\dfrac{1}{2}\int_{\mathbb{R}^{N}}\int_{\mathbb{R}^{N}}\dfrac{K(x)K(y)F(\sigma w_n(x))F(\sigma w_n(y))}{|x|^{\alpha}|x-y|^{\lambda}|y|^{\alpha}}dxdy\\
	=& \sigma C -\dfrac{1}{2}\int_{\mathbb{R}^{N}}\int_{\mathbb{R}^{N}}\dfrac{K(x)K(y)F(\sigma w_n(x))F(\sigma w_n(y))}{|x|^{\alpha}|x-y|^{\lambda}|y|^{\alpha}}dxdy
	\end{align*}}
If $w=0$, segue de \eqref{1.15'} or \eqref{5.3} that
\begin{align*}
\lim_{n \rightarrow \infty}\int_{\mathbb{R}^{N}}\int_{\mathbb{R}^{N}}\dfrac{K(x)K(y)F(\sigma w_n(x))F(\sigma w_n(y))}{|x|^{\alpha}|x-y|^{\lambda}|y|^{\alpha}}dxdy=0,
\end{align*}
therefore,
\begin{align*}
\liminf_{n\rightarrow\infty}J(t_nu_n)\geq C\sigma,\;\;\;\; \forall\;\sigma\geq 1.
\end{align*}
which constitutes a contradiction with Lemma \ref{5}, once that $(J(t_n u_n))$ is bounded from above. 

\noindent{\it{Case:} $w\neq0$}

Recalling that 
\begin{align*}
|u_n|=|w_n|\lVert u_n\lVert_{E} \;\;\text{ and }\;\;
\dfrac{u_n(x)}{\lVert u_n\lVert_{E}}=w_n(x)\longrightarrow w(x),\;a.e.\, \text{ in } \mathbb{R}^{N}
\end{align*}
we will get that
\begin{align*}
|w_n(x)|\longrightarrow |w(x)|,\;\;a.e.\, \text{ in } \mathbb{ R}^N.
\end{align*}
Furthermore, from the fact that $\lVert  u_n\lVert_{E}\rightarrow+\infty$, we can conclude that
\begin{align*}
|u_n(x)|=|w_n(x)|\lVert  u_n\lVert_{E} \rightarrow+\infty,\;\text{ as }\;n\rightarrow\infty \;\text{ for } x\in\{y\in\mathbb{ R}^N: w(y)\neq 0\}.
\end{align*}
 By \eqref{11.39},
\begin{align}\label{1.24}
0=\limsup_{n \to \infty}\dfrac{c}{\lVert u_n\lVert_{E}^m}=\limsup_{n \to \infty}\dfrac{J(u_n)}{\lVert u_n\lVert_{E}^m}.
\end{align}
 As $\lVert u_n\lVert_{\Phi}\geq 1$ and $\lVert u_n\lVert_{V,\Phi}\geq 1$ for every $n\geq n_0$,
\begin{align}\label{1.26}
\int_{\mathbb{R}^{N}}\Phi(|\nabla u_n|)dx\leq \lVert \nabla u_n\lVert_{\Phi}^{m}\;\text{ and }\;\int_{\mathbb{R}^{N}}V(x)\Phi(|u_n|) dx\leq \lVert u_n\lVert_{V,\Phi}^{m},\;\forall n\geq n_0.
\end{align}
Thus, it follows from $(f_4)$, \eqref{1.24}, \eqref{1.26} and Fatou’s Lemma that
\begin{align*}
0=&\limsup_{n \to \infty}\dfrac{J(u_n)}{\lVert u_n\lVert_{E}^m}\\
\leq &\limsup_{n\rightarrow\infty} \left[ \dfrac{1}{\lVert u_n\lVert_{E}^{m}}\int_{\mathbb{R}^{N}}\Phi(|\nabla u_n|)dx+\dfrac{1}{\lVert u_n\lVert_{E}^{m}}\int_{\mathbb{R}^{N}}V(x)\Phi(|u_n|)dx\right]\\
&-\liminf_{n\rightarrow\infty} \left[ \dfrac{1}{2}\int_{\mathbb{R}^{N}}\int_{\mathbb{R}^{N}}\dfrac{K(x)K(y)}{|x|^{\alpha}|x-y|^{\lambda}|y|^{\alpha}}\dfrac{F(u_n(x))F(u_n(y))}{\lVert u_n\lVert_{E}^m}dxdy\right]\\
\leq &2-\dfrac{1}{2}\liminf_{n\rightarrow\infty} \left[ \int_{\mathbb{R}^{N}}\int_{\mathbb{R}^{N}}\dfrac{K(x)K(y)}{|x|^{\alpha}|x-y|^{\lambda}|y|^{\alpha}}\dfrac{F(u_n(x))}{|u_n(x)|^{\frac{m}{2}}}|w_n(x)|\frac{F(u_n(y))}{|u_n(y)|^{\frac{m}{2}}}|w_n(y)|dxdy\right]\\
=&-\infty
\end{align*}
which is a contradiction. This shows that $(u_n)$ is bounded in $E$.

The cases $i)$ and $ii)$ are analogous to the case $iii)$.

\qed

\section{Ground state solutions for $(P)$}

Since that the Cerami sequence $(u_n)$ given in \eqref{11.39} is bounded in  $E$, by Lemma \ref{0.6}, we can assume that for some subsequence, there is $u \in   E$ such that
\begin{align}\label{6.4}
u_{n}\xrightharpoonup[\quad]{\ast} u\;\;\text{ in }D^{1,\Phi}(\mathbb{R}^{N})\;\;\;\text{ and }\;\;\;
u_{n}(x)\longrightarrow u(x)\;\;a.e.\;\;\mathbb{R}^{N}.
\end{align}

Fix $v\in C^{\infty}_0(\mathbb{R}^N)$. By boundedness of Cerami sequence $(u_n)$, we have $J'(u_n)(v-u_n)=o_n(1)$, hence, since $\Phi$ is a convex function, it is possible to show that
\begin{align}\label{7.2}
\begin{split}
Q(v)-Q(u_n)\geq\int_{\mathbb{R}^{N}} \int_{\mathbb{R}^{N}}\dfrac{K(x)K(y)F(u_n(x))f(u_n(y))(v(y)-u_n(y))}{|x|^{\alpha}|x-y|^{\lambda}|y|^{\alpha}}dxdy+o_n(1).
\end{split}
\end{align}
For the Lemma \ref{0.1}, we have to
\begin{align*}
\liminf_{n\rightarrow\infty}\int_{\mathbb{R}^{N}}\Phi(|\nabla u_n|)dx\geq\int_{\mathbb{R}^{N}}\Phi( |\nabla u|)dx.
\end{align*}
Now, due to the Fatou's Lemma, we have
\begin{align*}
\liminf_{n\rightarrow\infty}\int_{\mathbb{R}^{N}}V(x)\Phi(| u_n|)dx\geq\int_{\mathbb{R}^{N}}V (x)\Phi(|u|)dx.
\end{align*}
Therefore,
\begin{align}\label{7.3}
\liminf_{n\rightarrow\infty} Q(u_n)\geq Q(u).
\end{align}
From \eqref{7.2} and \eqref{7.3} together with the limits \eqref{1.15'} and \eqref{5.9} (or the limit \eqref{11.113'} and \eqref{5.10}), we get
\begin{align*}
Q(v)-Q(u)\geq\int_{\mathbb{R}^{N}}\int_{\mathbb{R}^{N}}\dfrac{K(x)K(y)F(u(x))f(u(y))(v(y)-u(y))}{|x|^{\alpha}|x-y|^{\lambda}|y|^{\alpha}}dxdy.
\end{align*}
As $E=\overline{C^{\infty}_{0}(\mathbb{R}^{N})}^{\lVert\cdot\lVert_{E}}$ and $\Phi\in( \Delta_{2})$, we conclude that
\begin{align}\label{07.4}
Q(v)-Q(u)\geq\int_{\mathbb{R}^{N}}\int_{\mathbb{R}^{N}}\dfrac{K(x)K(y)F(u(x))f(u(y))(v(y)-u(y))}{|x|^{\alpha}|x-y|^{\lambda}|y|^{\alpha}}dxdy,\;\;\;\;\forall\;v\in E.
\end{align}
In other words, $u$ is a critical point of the $J$ functional. From Proposition \ref{0.4}, we can conclude that $u$ is a weak solution for $(P)$. Now, we substitute $v=u^+:=\max\{0,u(x)\}$ in \eqref{07.4} and we get
\begin{align*}
-\int_{\mathbb{R}^{N}}\Phi(|\nabla u^-|)dx-\int_{\mathbb{R}^{N}}V(x)\Phi(u^-)dx\geq\int_{\mathbb{R}^{N}}\int_{\mathbb{R}^{N}}\dfrac{K(x)K(y)F(u(x))f(u(y))(u^-(y))}{|x|^{\alpha}|x-y|^{\lambda}|y|^{\alpha}}dxdy=0,
\end{align*}
which leads to
\begin{align*}
\int_{\mathbb{R}^{N}}\Phi(|\nabla u^-|)dx=0\;\;\text{ and }\;\;\int_{\mathbb{R}^{N}}V(x)\Phi(u^-)dx=0
\end{align*}
whence it is readily inferred that $u^-=0$, therefore, $u$ is a weak nonnegative solution.

Note that $u$ is nontrivial. In the sense, consider a sequence $(\varphi_k)\subset C^{\infty}_{0}(\mathbb{R}^{N})$ such that $\varphi_k \rightarrow u$ in $D^{1 ,\Phi}(\mathbb{R}^{N})$. Since $(u_n)$ is bounded, we get $J'(u_n)(\varphi_k-u_n)=o_n(1)\lVert \varphi_k\lVert-o_n(1)$. As $\Phi$ is convex, we can show that
\begin{align}\label{5.11}
Q(\varphi_k)-Q(u_n)-o_n(1)\geq\int_{\mathbb{R}^{N}}\int_{\mathbb{R}^{N}}\dfrac{K(x)K(y)F(u_n(x))f(u_n(y))(\varphi_k(y)-u_n(y))}{|x|^{\alpha}|x-y|^{\lambda}|y|^{\alpha}}dxdy.
\end{align}
Since $(\lVert \varphi_k\lVert)_{k\in\mathbb{N}}$ is a bounded sequence, it follows from \eqref{5.11} and from limits \eqref{1.16'} and \eqref{5.9} (or the limit \eqref{11.113'} and \eqref{5.10})  that
\begin{align*}
Q(\varphi_k)- \limsup_{n \rightarrow \infty} Q(u_n)\geq \limsup_{n \rightarrow \infty} \int_{\mathbb{R}^{N}}\int_{\mathbb{R}^{N}}\dfrac{K(x)K(y)F(u_n(x))f(u_n(y))(\varphi_k(y)-u_n(y))}{|x|^{\alpha}|x-y|^{\lambda}|y|^{\alpha}}dxdy,
\end{align*}
for every $k\in\mathbb{ N}$. Now, note that being $\Phi\in (\Delta_{2})$ and $\varphi_k \rightarrow u$ in $E$, we conclude from the inequality above that
\begin{align} \label {7.8}
Q(u) \geq \limsup_{n \rightarrow \infty} Q(u_n).
\end{align}
From \eqref{7.3} and \eqref {7.8},
\begin{align} \label{7.9}
Q(u) = \lim_ {n \rightarrow \infty} Q(u_n).
\end{align}
By \eqref{1.15'} (or \eqref{5.3}), we have 
\begin{align}\label{7.21}
\lim_{n\rightarrow\infty} \int_{\mathbb{R}^{N}}\int_{\mathbb{R}^{N}}\dfrac{K(x)K(y)F(u_n(x))F(u_n(y))}{|x|^{\alpha}|x-y|^{\lambda}|y|^{\alpha}}dxdy=\int_{\mathbb{R}^{N}}\int_{\mathbb{R}^{N}}\dfrac{K(x)K(y)F(u_n(x))F(u_n(y))}{|x|^{\alpha}|x-y|^{\lambda}|y|^{\alpha}}dxdy,
\end{align}
Therefore
\begin{align*}
0<c=&\lim_{n\rightarrow\infty} J(u_n)\\
=&\int_{\mathbb{R}^{N}} \Phi(|\nabla u|)dx+\int_{\mathbb{R}^ {N}}V(x) \Phi(|u|)dx-\int_{\mathbb{R}^{N}}\int_{\mathbb{R}^{N}}\dfrac{K(x)K(y)F(u(x))F(u(y))}{|x|^{\alpha}|x-y|^{\lambda}|y|^{\alpha}}dxdy\\
=&J(u),
\end{align*}
that is, $u\neq 0$.

Now, we prove that the solution obtained is a
ground state solution. Let us recall the definition of a ground state solution:
\begin{definition}
	A weak solution $u\in E$ of $(P)$ is called a ground state solution if it has	the least energy, i.e., we say, the solution $u$ is ground state solution of $(P)$ if
	\begin{align}
		J(u)=b=\inf_{u\in \mathcal{S}} J(u)
	\end{align}
	where $\mathcal{S}$ is the set of all critical points of the functional $J$.
\end{definition}

In order to prove the result below, we will use the following continuity result:

\begin{lemma}\label{2.27}
	The function $u\mapsto J'(u)\cdot u$ is continuous from $E$ to $\mathbb R$.
\end{lemma}

The above lemma is immediate whenever $J \in C^{1}(E,\mathbb{ R})$.

\begin{lemma}\label{lem1}
	Assume that $(V, K) \in \mathcal{K}_1$ (or $(V, K) \in \mathcal{K}_2$) and $f$ satisfies $(f_1)-(f_4)$ ( or $(f_2), (f_3), (f_4), (f_5)$). For each $v\in E\setminus\{0\}$ the function $\psi_v(s)=J(sv)$ has the following properties:
	\begin{itemize}
		\item [$(\psi_1)$] there is a bounded closed interval $[a_v,b_v]$ (which can be degenerate) such that $0<a_v$ and $J'(sv)\cdot v>0$, for all $ s<a_v$
		\item [$(\psi_2)$] $0<\displaystyle \max_{s>0}J(sv)=J(\tau v)$, for all $\tau\in[a_v,b_v]$, $J (sv)>0$ in $s\in (0,a_v)$
		\item [$(\psi_3)$] $J(\tau v)<\displaystyle \max_{s>0}J(sv)$, for all $\tau\notin [a_v,b_v]$
		\item [$(\psi_4)$] There are $s_v>b_v$ and $\delta_v>0$ such that $J'(su)\cdot u<0$ and $J(su)<0$, for all $ s\geq s_v$ and $u\in B_{\delta_v}(v)$.
	\end{itemize}
\end{lemma}

\begin{proof}
Fixed $v\neq0$, the function $h(t)=J(tv)$ has derivative $h'(t)=J'(tv)\cdot v$. As in the Lemmas \ref{5.5} and \ref{5.5'}, there will be $r>0$ such that $J'(v)\cdot v>0$, for all $0<\|v\|\leq R$. Hence,
\begin{align}\label{2.25}
	J(sv)=\int_0^s J'(tv)\cdot vdt>0, \mbox { for } 0<s<\frac r{\|v\|}.
\end{align}
From $(f_4)$, there exists $s_v>0$ such that
\begin{align}\label{2.26}
J(sv)<0, \text{ for all } s>s_v
\end{align}
thus, $\displaystyle \max_{s>0}J(sv)=J(\tau v)>0$ for some $\tau\in (0,s_v)$. By definition of the $J$ functional, we have
{\small\begin{align*}
	\frac{J'(\tau v)v}{\tau ^{m-1}}=&\int_{\mathbb{R}^{N}} \frac{\phi(|\nabla tv|)|\nabla v|^2}{t^{m-2}} dx+\int_{\mathbb{R}^{N} } V(x)\frac{\phi(|tv|)|v|^2}{t^{m-2}} dx-\int_{\mathbb{R}^{N}}\int_{\mathbb{R}^{N}}\dfrac{K(x)K(y)F(tv(x))f(tv(y))v(y)}{t^{{\frac{m}{2}}}|x|^{\alpha}|x-y|^{\lambda}|y|^{\alpha}t^{\frac{m}{2}-1}}dxdy.
	\end{align*}}
Using the hypothesis $(\phi_4)$, we can conclude that the function $$t\mapsto\int_{\mathbb{R}^{N}} \frac{\phi(|\nabla tv|)|\nabla v| ^2}{t^{m-2}} dx+\int_{\mathbb{R}^{N} } V(x)\frac{\phi(|tv|)|v|^2}{t^{ m-2}} dx$$ is nonincreasing, since the hypothesis $(f_2)$ guarantees that the function $$ t \mapsto \int_{\mathbb{R}^{N}}\int_{\mathbb{R }^{N}}\dfrac{K(x)K(y)F(tv(x))f(tv(y))v(y)}{t^{{\frac{m}{2}} }|x|^{\alpha}|x-y|^{\lambda}|y|^{\alpha}t^{\frac{m}{2}-1}}dxdy$$ is nondecreasing. Therefore, by \eqref{2.25} and \eqref{2.26} there will be an interval $[a_v,b_v]$ such that $h'(\tau)>0$ in $t<a_v$, $h'(\tau )<0$ in $\tau>b_v$ and $h'(\tau)=0$ in the interval $[a_v,b_v]$. The conclusion $(\psi_1)$, $(\psi_2)$ and $(\psi_3)$ is immediate. The property $(\psi_{4})$ follows from the previous items together with \eqref{2.26} and with the continuity of $J$ and $v \mapsto J'(v)v$.

\end{proof}

 In the proof of the lemma below, we have adapted the ideas presented by Willem, which can be found in Theorem $4.2$ in \cite{Willem}.
 \begin{proposition}
 	If $u\in E$ is a nontrivial solution for $(P)$ such that $J(u)=c$, where $c$ is the level given in \eqref{11.39}. Then $c=\inf_{u\in \mathcal{S}} J(u)$
 	where $\mathcal{S}$ is the set of all critical points of the functional $J$.
 \end{proposition}
\begin{proof}
By condition $(f_4)$, we can fix without losing generality $e\in E$ such that $J(e)<0$ and $J'(e)\cdot e<0$. Consider the following sets:
	\begin{align*}
		\Gamma=\{\gamma:[0,1]\to E: \gamma(0)=0,\;\gamma(1)=e\},\;\;
		\Gamma_0=\{\gamma:[0,1]\to E: \gamma(0)=0,\;J(\gamma(1))<0\}	\end{align*}	\begin{align*}
		\mathcal N =\{v\in E\setminus\{0\}: J'(v)\cdot v=0\},\;\;
		\mathcal S =\{v\in E\setminus\{0\}: J'(v)=0\}.
	\end{align*}
We will compare the following numbers:
	\begin{align*}
		c=\inf_{\gamma\in	\Gamma}\max_{t\in[0,1]}J(\gamma(t)) ,\;\;	c_0=\inf_{\gamma\in	\Gamma_o}\max_{t\in[0,1]}J(\gamma(t)),\;\;d=\inf_{v\neq 0}\max_{s>0}J(sv),\;\;
		a=\inf_{v\in\mathcal N}J(v),\;\;b=\inf_{v\in\mathcal S}J(v).
	\end{align*}
Let us see some immediate inequalities:
\begin{itemize}
	\item [$(i)$] It is obvious that $c_0\leq c$ and $a\leq b$;
	\item [$(ii)$] Let us see $c_0\leq d$. Note that if $v\neq 0$, then the path $\gamma(t)=ts_vv$ is such that $\gamma(0)=0$ and $J(\gamma(1))=J(s_vv)< 0$. Therefore, $\gamma\in \Gamma_0$ and
	\[
	c_0\leq \max_{t\in[0,1]}J(\gamma(t))=\max_{t>0}J(sv),
	\]
	that is, $c_0$ is a lower bound for the definition of $d$. The affirmation is justified.
		\item [$(iii)$] Let us show that $a\leq c$. In fact, fix $\gamma\in \Gamma$. Just check that $J'(\gamma(1))\cdot \gamma(1)=J'(e)\cdot e<0<J'(\gamma(t))\cdot \gamma(t)$ for $t>0$ small enough. Having the Lemma \ref{2.27} true, we can use the Intermediate Value Theorem to guarantee the existence of $t_1\in(0,1)$ such that $J'(\gamma(t_1))\cdot \gamma(t_1 )=0$, so $ \gamma(t_1)\in \mathcal N$. Thus $a\leq J(\gamma(t_1))\leq \displaystyle\max_{t\in[0,1]}J(\gamma(t))<c+\varepsilon$, and therefore the inequality $a\leq c$ is shown.
		\item [$(iv)$] Now $c\leq c_0$. Let $\gamma\in \Gamma$. The idea is to define a function $\tilde\gamma\in \Gamma $ such that
$$
		\max_{t\in[0,1]}J(\gamma(t))=\max_{t\in[0,1]}J(\tilde\gamma(t)).$$
For this, define $\tilde\gamma:[0,1]\to E$ as follows:
$$
		\tilde\gamma(t)=\gamma (2t), \mbox { for }t\in [0,\frac12].
	$$
It remains to define the function $\tilde\gamma$ to values $t\in\displaystyle\left[\frac12,1\right]$. Remember that $J(e)<0$ and $J(\tilde\gamma(\frac12))=J(\gamma(1))<0$. Being $\gamma(1)=v_1$, consider any point $u$ of the segment $[e,v_1]$. We cover this compact segment with a finite number of balls $B_{\delta_u}(u)$ obtained through the property $(\psi_4)$, that is, $[e,v_1]\subset B_{\delta_{u_1}} (u_1)\cup B_{\delta_{u_2}}(u_2)\cup \cdots \cup B_{\delta_{u_n}}(u_n)$. Consider $\lambda=\max\{s_{u_1}, s_{u_2},\cdots, s_{u_2}\}$, numbers given by \eqref{2.26} and $(\psi_4)$. Set the $\tilde \gamma:\left[\displaystyle\frac12,1\right]\to E$ the polygonal line segment from $v_1$ going to $\lambda v_1$, then connecting $\lambda v_1$ to $\lambda e$ and finally, connecting $\lambda e$ to the point $e$. It is easily shown that $\tilde\gamma\in \Gamma$, $J(\tilde\gamma(t))<0$ for all $t\in\left[\displaystyle\frac12,1\right]$ and therefore, $\displaystyle \max_{t\in[0,1]}J(\gamma(t))=\max_{t\in[0,1]}J(\tilde\gamma(t))$. Showing that $c_0\geq c$.
		
		
		\item [$(v)$] Let us see that $d\leq a$. Fix $v\in \mathcal N$ so that $J(v)<a+\varepsilon$. In the proof of the Lemma \ref{lem1} the function defined by $\psi_v(t)=J(tv)$ satisfies $\psi_v'(t)=0$ only if $t\in [a_v,b_v]$. Consider $v\in \mathcal{N}$ and note that $\psi_v'(1)=J'(v).v=0$, this implies that $1\in [a_v,b_v]$. Knowing that the function $\psi_v$ reaches a maximum in the interval $[a_v,b_v]$, we will obtain $\psi_v(t)\leq \psi_v(1)$ for all $t>0$, because $\psi_v$ is constant in $[a_v,b_v]$. In light of this,
		\[d\leq \max_{s>0}J(sv)\leq J(v)<a+\varepsilon.\]
		If $\varepsilon$ is arbitrary, we have $d\leq a$.
	\end{itemize}		

	Finally, consider $u\in \mathcal S$ satisfying $J(u)=c$. By the inequalities above we can conclude that $a=b=c=c_o=d$.

\end{proof}

\subsection{Proof of Theorem \ref{Teo1} and \ref{Teo1.1}}\label{sec1} Assuming the assumptions of Theorem \ref{Teo1}, the above argument guarantees the existence of a nonnegative ground state solution for problem $(P)$, thus showing the first part of Theorem \ref{Teo1}. Now, to show to study the boundedness of nonnegative solutions of the problem $(P)$ we will make heavy use of hypothese $2\alpha+\lambda<2\ell$.

Now, we begin by presenting a technical result, which is an adaptation of a result that can be found in \cite{Ailton}.
\begin{lemma}\label{22.6}
	Let  $u\in E$ be a nonnegative solution of $(P)$, $x_0 \in  \mathbb{R}^N$ and $R_0>0$. Then 
	{\footnotesize\begin{align*}
	\int_{\mathcal{A}_{k,t}} |\nabla u|^\ell dx\leq C \left( \int_{\mathcal{A}_{k,s}}\left|\dfrac{u-k}{s-t}\right|^{\ell^*}dx +(k^{\ell^*}+1) |\mathcal{A}_{k,s}|\right)+ C \left( \int_{\mathcal{A}_{k,s}}\left|\dfrac{u-k}{s-t}\right|^{\ell^*}dx +(k^{\ell^*}+1) |\mathcal{A}_{k,s}|\right)^{\frac{1}{s}}
	\end{align*}}
	where $0<t<s<R_0$, $k>1$, $ \mathcal{A}_{k,\rho}=\{x\in B_{\rho}(x_0): u(x)>k\}$ and $C>0$ is a constant that does not depend on $k$.
\end{lemma}

\noindent{\bf{Proof:}}
Let $u\in E$ be a weak solution nonnegative of $(P)$ and $x_0\in \mathbb{R}^{N}$. Moreover, fix $0 < t < s < R_0 $ and $\zeta \in C^{\infty}_0(\mathbb{R}^N)$ verifying 
\begin{align*}
0\leq \zeta \leq 1, \quad supp (\zeta)\subset B_{s}(x_0),\quad \zeta\equiv 1\;\text{ on }\; B_t(x_0)\;\;\text{ and }\;\; |\nabla \zeta|\leq \dfrac{2}{s-t}.
\end{align*}
For $k >1$, set $\varphi = \zeta^m(u-k)^{+}$ and
\begin{align*}
J=\int_{\mathcal{A}_{k,s}}\Phi(|\nabla u|)\zeta ^mdx,
\end{align*}
Using $\varphi$ as a test function and $\ell \Phi(t)\leq \phi(t)t^2$, we find
%
\begin{align*}
\ell J\leq &m \int_{\mathcal{A}_{k,s}}\zeta^{m-1}(u-k)^+ \phi(|\nabla u|)|\nabla u| |\nabla\zeta| dx
-\int_{\mathcal{A}_{k,s}}V(x) \phi( u) u\zeta^m (u-k)^+ dx\\&+\int_{\mathcal{A}_{k,s}}\int_{\mathbb{R}^{N}}\dfrac{K(x)K(y)F(u(x))f(u(y))\zeta^m(y) (u(y)-k)^+}{|x-y|^{\lambda}}dxdy\\\leq&\int_{\mathcal{A}_{k,s}}\zeta^{m-1}(u-k)^+ \phi(|\nabla u|)|\nabla u| |\nabla\zeta| dx\\&+C_1(s,\lambda,N)
 \left|\int_{\mathbb{R}^{N}}K(x)^s|F(u)|^sdx\right|^{\frac{1}{s}}\left|\int_{\mathbb{R}^{N}}W(x)|f(u)|^s|\zeta^m (u-k)^+|^sdx\right|^{\frac{1}{s}}
\end{align*}
By $(f_1)$, given $\eta > 0$, there exists $C_\varepsilon>0$ such that
\begin{align*}
K(x)^sf(t)^s\leq \frac{\varepsilon}{s}K(x)^sa(t)t^{2-s}+C_{\varepsilon}K(x)^sz(t)t^{2-s}, \;\;\;\forall t\geq0\text{ and } x\in\mathbb{R}^N.
\end{align*}
 Thus, 
\begin{align}\label{11.2}
\begin{split}
\ell J\leq& m \int_{\mathcal{A}_{k,s}}\zeta^{m-1}(u-k)^+ \phi(|\nabla u|)|\nabla u| |\nabla\zeta| dx
\\
&+C_2\left[\int_{\mathcal{A}_{k,s}}K^{s}(x) a(| u|) u^{2-s}(\zeta^m (u-k)^+)^s dx+\int_{\mathcal{A}_{k,s}}K^{s}(x) z(| u|) u^{2-s}(\zeta^m (u-k)^+)^s dx\right]^\frac{1}{s}
\end{split}
\end{align}
where $C_2=C_1(s,\lambda,N)
\left(\int_{\mathbb{R}^{N}}K(x)^s|F(u)|^sdx\right)^{\frac{1}{s}}$. For each $\tau \in(0, 1)$, the Young’s inequalities gives
%
%
\begin{align}\label{11.0}
\begin{split}
\phi(|\nabla u|)|\nabla u||\nabla \zeta| \zeta^{m-1}(u-k)^+\leq \tilde{\Phi }( 	\phi(|\nabla u|)|\nabla u|\zeta^{m-1}\tau)+C_3\Phi\Big(\Big|\dfrac{u-k}{s-t}\Big|\Big).
\end{split}
\end{align}
It follows from Lemma \ref{0.51}, 
\begin{align}\label{11.1}
\tilde{\Phi }( 	\phi(|\nabla u|)|\nabla u|\zeta^{m-1}\tau)\leq C_4(\tau \zeta^{m-1})^\frac{m}{m-1}{\Phi }(	|\nabla u|).
\end{align}	
From \eqref{11.2}, \eqref{11.0} and \eqref{11.1}, 
\begin{align*}
\ell J\leq& m  C_4\tau^\frac{m}{m-1}\int_{\mathcal{A}_{k,s}} \Phi (	|\nabla u|) \zeta^{m}+mC_3\int_{\mathcal{A}_{k,s}} \Phi\Big(\Big|\dfrac{u-k}{s-t}\Big|\Big)dx\\
&+C_2\left[\int_{\mathcal{A}_{k,s}}K^{s}(x) a(| u|) u^{2-s}(\zeta^m (u-k)^+)^s dx+\int_{\mathcal{A}_{k,s}} K^{s}(x) z(| u|) u^{2-s}(\zeta^m (u-k)^+)^s dx\right]^\frac{1}{s}
\end{align*}
Choosing $\tau\in (0,1)$  such that $0<m  C_4\tau^\frac{m}{m-1}<\ell$, we derive
{\footnotesize\begin{align}\label{11.6}
J\leq C_5\int_{\mathcal{A}_{k,s}} \Phi\Big(\Big|\dfrac{u-k}{s-t}\Big|\Big)dx+C_2\left[\int_{\mathcal{A}_{k,s}}K^{s}(x) a(| u|) u^{2-s}(\zeta^m (u-k)^+)^s dx+\int_{\mathcal{A}_{k,s}} K^{s}(x) z(| u|) u^{2-s}(\zeta^m (u-k)^+)^s dx\right]^\frac{1}{s}
\end{align}}
By Young’s inequalities,
\begin{align}\label{11.7}
z( u)u^{2-s}(\zeta^m (u-k)^+)^s \leq C_6Z\left(\Big|\dfrac{u-k}{s-t}\Big| \right)+ C_6Z(k).
\end{align}
and
\begin{align}\label{11.7'}
a( u)u^{2-s}(\zeta^m (u-k)^+)^s \leq C_6A\left(\Big|\dfrac{u-k}{s-t}\Big| \right)+ C_6A(k).
\end{align}
Therefore, a combination of \eqref{11.6} and \eqref{11.7}, yields 
{\footnotesize\begin{align}\label{111}
J\leq C_{7}\int_{\mathcal{A}_{k,s}} \Phi\Big(\Big|\dfrac{u-k}{s-t}\Big|\Big)dx+C_{7}\left[\int_{\mathcal{A}_{k,s}}A\Big(\Big|\dfrac{u-k}{s-t}\Big|\Big)dx+ \int_{\mathcal{A}_{k,s}}A(k)dx+\int_{\mathcal{A}_{k,s}} Z\Big(\Big|\dfrac{u-k}{s-t}\Big|\Big)dx+ \int_{\mathcal{A}_{k,s}}Z(k)dx\right]^{\frac{1}{s}}.
\end{align}}
 Now, using that $\ell\leq m<a_2<\ell^* $ and applying the Lemmas \ref{0.3}, \ref{0.5} and the Remark \ref{5.02} for functions $\Phi$, $A$ and $\Phi_{*},$ respectively, we get
\begin{align*}
\Phi\Big(\Big|\dfrac{u-k}{s-t}\Big|\Big)\leq \Phi(1)\Big( \Big|\dfrac{u-k}{s-t}\Big|^{\ell^*}+1 \Big),
\end{align*}
\begin{align*}
A\Big(\Big|\dfrac{u-k}{s-t}\Big|\Big)\leq A(1)\Big( \Big|\dfrac{u-k}{s-t}\Big|^{\ell^*}+1 \Big)\text{ and }A(k)\leq(k^{\ell^*}+1).
\end{align*}
and
\begin{align*}
Z\Big(\Big|\dfrac{u-k}{s-t}\Big|\Big)\leq Z(1)\Big( \Big|\dfrac{u-k}{s-t}\Big|^{\ell^*}+1 \Big)\text{ and }Z(k)\leq(k^{\ell^*}+1).
\end{align*}
From \eqref{111} and the inequality above,
\begin{align*}
J\leq C_{8}\left(\int_{\mathcal{A}_{k,s}} \Big|\dfrac{u-k}{s-t}\Big|^{\ell^*}dx + (k^{\ell^*}+1) |\mathcal{A}_{k,s}|\right)+C_{8}\left(\int_{\mathcal{A}_{k,s}} \Big|\dfrac{u-k}{s-t}\Big|^{\ell^*}dx + (k^{\ell^*}+1) |\mathcal{A}_{k,s}|\right)^{\frac{1}{s}}.
\end{align*}

\qed

\begin{lemma}\label{22.21}
	Let $u \in E$ be a nonnegative solution of $(P)$. Then, $u \in L^{\infty}_{loc} (\mathbb{R}^N)$.
\end{lemma}
\noindent{\bf{Proof:}} To begin with, consider $\Lambda$ a compact subset on $\mathbb{R}^N$. Fix $R_1\in (0,1)$, $x_0\in\Lambda$ and define the sequences
\begin{align*}
\sigma_n=\dfrac{R_1}{2}+\dfrac{R_1}{2^{n+1}}\;\;\text{ and }\;\;\overline{\sigma}_n=\dfrac{\sigma_n +\sigma_{n+1}}{2}\;\;\text{ for }n=0,1,2,\cdots.
\end{align*}
Note that
\begin{align*}
\sigma_n\downarrow\dfrac{R_1}{2}\quad\text{ and }\quad \sigma_{n+1}<\overline{\sigma}_n <\sigma_n<R_1.
\end{align*}
 
Since $E$ is continuously embedded in $W^{1,\ell}_{loc}(\mathbb{R}^N)$, it follows from Lebesgue dominated convergence theorem,
\begin{align}
\lim_{M\rightarrow\infty} \int_{B_{R_1}(x_0)}\big((u-M)^+\big)^{\ell^*}dx=0,
\end{align}
hence, there is $M^*\geq 1$ which depends on $x_0$ and $R_1$, such that
\begin{align}\label{5.6}
\int_{B_{R_1}(x_0)}\big((u-M)^+\big)^{\ell^*}dx\leq 1 , \;\;\text{ for }\; M\geq M^*.
\end{align}
Now, consider $K>4M^*$ and for every $n\in\mathbb{N}$ define
\begin{align*}
K_n=\dfrac{K}{2}\left(1-\dfrac{1}{2^{n+1}}\right)\;\;\text{ and }\;\;J_n=\int_{ \mathcal{A}_{K_n,\sigma_n}}\big((u-K_n)^+\big)^{\ell^*}dx,\;\text{ for }n=0,1,2,\cdots.
\end{align*}
and
\begin{align*}
\xi_n=\xi\left(\dfrac{2^{n+1}}{R_1}\Big(|x-x_0|-\dfrac{R_1}{2}\Big)\right),\;\; x\in\mathbb{R}^N\text{ and }n=0,1,2,\cdots
\end{align*}
where $\xi\in C^{1}(\mathbb{R})$ satisfies
\begin{align*}
0\leq \xi\leq 1,\quad \xi(t)=1\;\text{ for }\; t\leq \dfrac{1}{2}\quad\text{ and }\quad \xi(t)=0\;\text{ for }\; t\geq \dfrac{3}{4}..
\end{align*}
From definition of  $\xi_{n}$,
\begin{align*}
\xi_{n}=1\;\;\text{ in }\;B_{\sigma_{{n}+1}}(x_0)\quad\text{ and }\quad\xi_{n}=0\;\text{ outside }\;B_{\overline{\sigma}_{n}}(x_0),
\end{align*}
consequently
\begin{align}\label{22.23}
\begin{split}
J_{n+1}&\leq\int_{B_{R_1}(x_0)}\big((u-K_{n+1})^+\xi_{n}\big)^{\ell^*}dx\\
&\leq C_1\left(\int_{\mathcal{A}_{K_{n+1},\overline{\sigma}_{n}}}|\nabla((u-K_{+1})^+\xi_{n})|^{\ell}dx\right)^{\frac{\ell^{*}}{\ell}}\\&\leq C_2\left(\int_{\mathcal{A}_{K_{{n+1}},\overline{\sigma}_{n}}}|\nabla u|^\ell dx+ 2^{\ell{n}}\int_{\mathcal{A}_{K_{n+1}},\overline{\sigma}_{n}}((u-K_{{n+1}})^+)^{\ell}dx\right)^{\frac{\ell^{*}}{\ell}}
\end{split}
\end{align}
for some constant $C_2=C(N,\ell,R_1)>0$. Applying the Lemma \ref{22.6} to the previous inequality, we get
{\footnotesize\begin{align}\label{22.5}
\begin{split}
J_{{n+1}}^{\frac{\ell}{\ell^{*}}}\leq &C_3\Big(\int_{\mathcal{A}_{K_{n+1},{\sigma_{n}}}}\Big|\dfrac{u-K_{{n_k}+1}}{\sigma_{n}-\overline{\sigma_{n}}}\Big|^{\ell^*}dx +(K_{n+1} ^{\ell^*}+1)|\mathcal{A}_{K_{{n}+1},{\sigma_{n}}}|        +2^{\ell{n}}\int_{\mathcal{A}_{K_{{n}+1},\overline{\sigma}_{n}}}((u-K_{{n}+1})^+)^{\ell}dx\Big)\\
&+C_3\Big(\int_{\mathcal{A}_{K_{{n+1}},{\sigma_{n}}}}\Big|\dfrac{u-K_{{n}+1}}{\sigma_{n}-\overline{\sigma_{n}}}\Big|^{\ell^*}dx +(K_{{n}+1} ^{\ell^*}+1)|\mathcal{A}_{K_{{n}+1},{\sigma_{n}}}| \Big)^{\frac{1}{s}}
\end{split}
\end{align}}
where $C_3>0$ is a constant that depends only on $N$, $\ell$ and $R_1$. Being $|\sigma_{n} -\overline{\sigma}_{n}| =\dfrac{R_1}{2^{{n}+3}}$, we conclude that
{\footnotesize\begin{align}\label{22.24}
\begin{split}
J_{{n+1}}^{\frac{\ell}{\ell^{*}}}\leq C_4(N,\ell,R_1)\Big(&2^{\ell{n}}\int_{\mathcal{A}_{K_{{n}+1},{\sigma}_{n}}}((u-K_{{n}+1})^+)^{\ell^*}dx +(K^{\ell^*}+1)|\mathcal{A}_{K_{{n}+1},{\sigma_{n}}}| \\
&       +2^{\ell{n}}\int_{\mathcal{A}_{K_{{n}+1},\overline{\sigma}_{n}}}((u-K_{{n}+1})^+)^{\ell}dx\Big)\\
+ C_4(N,\ell,R_1)\Big(&2^{\ell{n}}\int_{\mathcal{A}_{K_{{n}+1},{\sigma_{n}}}}((u-K_{{n}+1})^+)^{\ell^*}dx +(K^{\ell^*}+1)|\mathcal{A}_{K_{{n}+1},{\sigma_{n}}}|\Big)^{\frac{1}{s}}
\end{split}
\end{align}}
Combined the inequality above with $t^\ell \leq t^{\ell^*} +1$, for $t \geq 0$ and using that $\overline{\sigma}_{n} <\sigma_{n}$, we get that
\begin{align}\label{22.9}
\begin{split}
J_{{n}+1}^{\frac{\ell}{\ell^{*}}}\leq C_4(N,\ell,R_1)\Big(&2^{\ell{n}}\int_{\mathcal{A}_{K_{{n}+1},{\sigma_{n_k}}}}((u-K_{{n}+1})^+)^{\ell^*}dx +(K^{\ell^*}+2^{\ell{n}}+1)|\mathcal{A}_{K_{{n}+1},{\sigma_{n}}}|\Big)\\
+ C_4(N,\ell,R_1)\Big(&2^{\ell{n}}\int_{\mathcal{A}_{K_{{n}+1},{\sigma_{n}}}}((u-K_{{n}+1})^+)^{\ell^*}dx +(K^{\ell^*}+2^{\ell{n}}+1)|\mathcal{A}_{K_{{n}+1},{\sigma_{n}}}|\Big)^{\frac{1}{s}}
\end{split}
\end{align}
On the other hand, since $K_{n+1}-K_{{n}}=\dfrac{K}{2^{{n}+3}}$,
\begin{align}\label{22.10}
\begin{split}
\left(\dfrac{K}{2^{{n}+3}}\right)^{\ell^*}\big|\mathcal{A}_{K_{{n}+1},{\sigma_{n}}}\big|&=(K_{{n}+1}-K_{{n}})^{\ell^*}\big|\mathcal{A}_{K_{{n}+1},{\sigma_{n}}}\big|\\
&\leq \int_{\mathcal{A}_{K_{{n}+1},{\sigma_{n}}}}(K_{{n}+1}-K_{{n}})^{\ell^*}dx\\
&\leq\int_{\mathcal{A}_{K_{{n}+1},{\sigma_{n}}}} \Phi_*((u-K_{n})^+)\chi_{\mathcal{A}_{K_{{n}+1},{\sigma_{n}}}}(x)\leq J_{{n}},
\end{split}
\end{align}
which yields
\begin{align}\label{22.12}
\begin{split}
\big|\mathcal{A}_{K_{{n}+1},{\sigma_{n}}}\big|\leq\dfrac{1}{\left(\dfrac{K}{2^{{n}+3}}\right)^{\ell^*}}J_{{n}}.
\end{split}
\end{align}
Thus,
\begin{align*}
\int_{\mathcal{A}_{K_{{n}+1},{\sigma_{n}}}}\big((u-K_{{n}+1})^+\big)^{\ell^*}dx&\leq \int_{\mathcal{A}_{K_{{n+1}},{\sigma_{n}}}}\big((u-K_{{n}})^+\big)^{\ell^*}dx+\int_{\mathcal{A}_{K_{{n}+1},{\sigma_{n}}}}\big(K_{{n}+1}-K_{n}\big)^{\ell^*}dx\\
&\leq \int_{\mathcal{A}_{K_{{n}},{\sigma_n}}}\big((u-K_{{n}})^+\big)^{\ell^*}dx+\big|K_{{{n+1}}}-K_{{n}}\big|^{\ell^*}\big|\mathcal{A}_{K_{{n}+1},{\sigma_{n}}}\big|\\
&\leq 2J_{n}
\end{align*}
and consequently
\begin{align}\label{22.25}
\begin{split}
J_{{n}+1}^{\frac{\ell}{\ell^{*}}}\leq& C_5(N,\ell,R_1)\Big(2^{\ell n+1}+2^{n(\ell+\ell^*)} +(2^{\ell n}+1)2^{n(\ell+\ell^*)}\Big)J_n\\
&+ C_5(N,\ell,R_1)\Big(2^{\ell n+1}+2^{n(\ell+\ell^*)} +(2^{\ell n}+1)2^{n(\ell+\ell^*)}\Big)^{\frac{1}{s}}J_n^{\frac{1}{s}}
\end{split}
\end{align}

Due to the fact that $K>4M^*$, we conclude from the inequality \ref{5.6} that
\begin{align*}
	J_n\leq \int_{B_{R_1}(x_0)}\big((u-M^*)^+\big)^{\ell^*}dx\leq 1,\;\text{ for } n=0,1,2,\cdots,
\end{align*}
hence,
\begin{align}\label{22.19}
J_{{n+1}}
\leq CD^{{n}}J_{{n}}^{1+\omega},
\end{align}
where $C=2C_5(N,\ell,R_1),$ $D=2^{2(\ell+\ell^*){\frac{\ell^{*}}{\ell}}}$ and $\omega=\frac{\ell^*}{s\ell}-1$.

We claim that
\begin{align*}
J_{0}\leq C^{-\frac{1}{\omega}} D^{-\frac{1}{\omega^2}}, \;\;\text{for }\; K\geq K^*.
\end{align*}
Indeed, note that,
\begin{align}\label{22.16}
\begin{split}
J_{0}&=\int_{\mathcal{A}_{K_{{0}},{\sigma_{0}}}}\big((u-K_{{0}})^+\big)^{\ell^*}dx\leq \int_{B_{R_1}(x_0)}\big((u-K_{{0}})^+\big)^{\ell^*}dx
\end{split}
\end{align}
Since $E$ is continuously embedded in $W^{1,\ell}_{loc}(\mathbb{R}^N)$, it follows from Lebesgue dominated convergence theorem,
\begin{align*}
\lim_{K\rightarrow\infty} \int_{B_{R_1}(x_0)}\big((u-K_{{0}})^+\big)^{\ell^*}dx=0.
\end{align*}
Therefore, there exists $K^*\geq 5M^*$ that depends on $x_0$, such that
\begin{align}\label{22.17}
\int_{B_{R_1}(x_0)}\big((u-K_{{0}})^+\big)^{\ell^*}dx\leq C^{-\frac{1}{\omega }} D^{-\frac{1}{\omega^2}} , \;\;\text{ for }\; K\geq K^*.
\end{align}
From \eqref{22.16} and \eqref{22.17},
\begin{align}\label{22.18}
J_{0}\leq C^{-\frac{1}{\omega}} D^{-\frac{1}{\omega^2}}, \;\;\text{for }\; K\geq K^*.
\end{align}
Fix $K = K^*$, by Lemma [\citenum{O.A}, Lemma 4.7], we deduce that
\begin{align*}
J_{n} \longrightarrow 0\;\;\text{ as }\;\;n\rightarrow \infty.
\end{align*}
On the other hand, 
\begin{align*}
\lim_{n\rightarrow\infty}J_{n}=\lim_{n\rightarrow\infty}\int_{\mathcal{A}_{K_{n},\sigma_{n}}}\big((u-K_{n})^+\big)^{\ell^*}dx=\int_{\mathcal{A}_{\frac{K^*}{2},\frac{R_1}{2}}}\big((u-\frac{K^*}{2})^+\big)^{\ell^*}dx
\end{align*}
Hence,
\begin{align*}
\int_{\mathcal{A}_{\frac{K^*}{2},\frac{R_1}{2}}}\big((u-\frac{K^*}{2})^+\big)^{\ell^*}dx=0,
\end{align*}
leading to
\begin{align*}
u(x)\leq \frac{K^*}{2}\;\text{ a.e. }\text{ in } B_{\frac{R_1}{2}}(x_0).
\end{align*}
Since $x_0$ is arbitrary and $\Lambda$ is a compact subset, the last inequality ensures that
\begin{align}\label{5.12}
u(x)\leq \frac{M}{2}\;\text{ a.e. }\text{in } \Lambda
\end{align}
for some constant $M>0$. By the arbitrariness of $\Lambda$, we conclude that $u\in L^{\infty}_{loc}(\mathbb{R}^{N})$. 

\qed

This Lemmas guarantees that the Theorem \ref{Teo1} is valid. The proof of Theorem \ref{Teo1.1} will be divided into the following Lemmas:

\begin{lemma}
	$u\in C^{1,\alpha}_{loc} (\mathbb{ R}^N)$.
\end{lemma}

\begin{proof}
By $(f_1)$ together with the remark \ref{5.02}, there exists $C_1>0$ such that
\begin{align*}
|F(t)|^s\leq C_1\Phi_{*}(|t|),\;\;\forall t\geq1
\end{align*}
Hence, by Hölder inequality
{\footnotesize\begin{align*}
\int_{[u\geq1]} \dfrac{K(y)F(u(y))}{|x-y|^\lambda}dy\leq&\left(	\int_{[u\geq1]} \dfrac{K(y)^{\frac{s}{s-1}}}{|x-y|^{\frac{\lambda(s-1)}{s}}}dy\right)^{\frac{s-1}{s}}\left(\int_{[u\geq1]} {F(u(y))^s}dy\right)^{\frac{1}{s}}\\
\leq&C_1\left(	\int_{\mathbb{ R}^N} \dfrac{K(y)^{\frac{s}{s-1}}}{|x-y|^{\frac{\lambda(s-1)}{s}}}dy\right)^{\frac{s-1}{s}}\left(\int_{\mathbb{ R}^N} {\Phi_{*}(u)}dy\right)^{\frac{1}{s}}\\
\leq&C_1\left(\int_{\mathbb{ R}^N} {\Phi_{*}(u)}dy\right)^{\frac{1}{s}}\left(	\int_{|x-y|<1} \dfrac{K(y)^{\frac{s}{s-1}}}{|x-y|^{\frac{\lambda(s-1)}{s}}}dy+	\int_{|x-y|\geq1} \dfrac{K(y)^{\frac{s}{s-1}}}{|x-y|^{\frac{\lambda(s-1)}{s}}}dy\right)^{\frac{s-1}{s}}\\
\leq&C_1 \left(\int_{\mathbb{ R}^N} {\Phi_{*}(u)}dy\right)^{\frac{1}{s}}\left(	\lVert K\lVert_{\infty}^s\int_{|x-y|<1} \dfrac{1}{|x-y|^{\frac{\lambda(s-1)}{s}}}dy+	\int_{|x-y|\geq1} {K(y)^{\frac{s}{s-1}}}dy\right)^{\frac{s-1}{s}}\\
\leq&C_1\left(\int_{\mathbb{ R}^N} {\Phi_{*}(u)}dy\right)^{\frac{1}{s}}\left(	\lVert K\lVert_{\infty}^s\int_{0}^{1} {r^{N-1-\frac{\lambda(s-1)}{s}}}dy+	\lVert K\lVert_{ L^{\frac{s-1}{s}}(\mathbb{ R}^{N})}^{\frac{s-1}{s}}\right)^{\frac{s-1}{s}}\\
\end{align*}}
On the other hand, 
\begin{align*}
\int_{[u\leq1]} \dfrac{K(y)F(u(y))}{|x-y|^\lambda}dy\leq&\lVert F\lVert_{ L^{\infty}([0,1])}\int_{\mathbb{ R}^N} \dfrac{K(y)}{|x-y|^{\lambda}}dy\\
\leq&\lVert F\lVert_{ L^{\infty}([0,1])}\left(\int_{|x-y|<1} \dfrac{K(y)}{|x-y|^{\lambda}}dy+\int_{|x-y|\geq1} \dfrac{K(y)}{|x-y|^{\lambda}}dy\right)\\
\leq&\lVert F\lVert_{ L^{\infty}([0,1])}\left(	\lVert K\lVert_{\infty}\int_{0}^{1} {r^{N-1-\lambda}}dy+	\lVert K\lVert_{ L^{1}(\mathbb{ R}^{N})}\right),
\end{align*}
that is,
\begin{align*}
	\int_{\mathbb{ R}^N} \dfrac{K(y)F(u(y))}{|x-y|^\lambda}dy \leq C_2,
\end{align*}
where
\begin{align*}
	C_2=&\Big\{ C_1\left(\int_{\mathbb{ R}^N} {\Phi_{*}(u)}dy\right)^{\frac{1}{s}}\left(	\lVert K\lVert_{\infty}^s\int_{0}^{1} {r^{N-1-\frac{\lambda(s-1)}{s}}}dy+	\lVert K\lVert_{ L^{\frac{s-1}{s}}(\mathbb{ R}^{N})}^{\frac{s-1}{s}}\right)^{\frac{s-1}{s}},\\
	&\;\quad\quad\quad\lVert F\lVert_{ L^{\infty}([0,1])}\left(	\lVert K\lVert_{\infty}\int_{0}^{1} {r^{N-1-\lambda}}dy+	\lVert K\lVert_{ L^{1}(\mathbb{ R}^{N})}\right)\Big\},
\end{align*}
showing that
\begin{align}\label{5.13}
\int_{\mathbb{ R}^N} \dfrac{K(y)F(u(y))}{|x-y|^\lambda}dy\in L^{\infty}(\mathbb{ R}^N).
\end{align}
Consider the scalar measurable function $Z:\Omega\times \mathbb{R} \times \mathbb{R}^N \longrightarrow \mathbb{R}$ given by $$Z(x,t,p)=V(x)\varphi (|t|)t-\left(\int_{\mathbb {R}^{N}} \dfrac{K(y)F(u(y))}{|x-y|^\lambda}dy\right)K(x)\tilde{f}(t)$$
with
\begin{align*}
\varphi(t)=\left\{\begin{array}{lr}
\;\;\; \;\phi(t)\;\;\;, \;\;\;\;\;\;\text{ for } 0< t\leq M/2\;\;\;\vspace*{0.2cm}\\
\phi(M/2)\;,\text{ for } t\geq M/2
\end{array}\right.,
\end{align*}
and \begin{align*}
\varphi(t)=\left\{\begin{array}{lr}
\;\;\; \;f(t)\;\;\;, \;\;\;\;\;\;\text{ for } 0< t\leq M/2\;\;\;\vspace*{0.2cm}\\
f(M/2)\;,\text{ for } t\geq M/2
\end{array}\right.,
\end{align*}
where $M>0$ is the constant satisfying \eqref{5.12}. By \eqref{5.13}, there exists a constant $C_3$ such that
\begin{align*}
	|B(x,t,p)|\leq C_3
\end{align*}
This fact together with the hypothesis $(\phi_{6})$ allows us to apply the theorem of regularity due to Lieberman [\citen{O.A}, Theorem 1.7]. Thus showing the result

\end{proof}

\begin{corollary}\label{22.2'}
	Let $u\in E$ be a nonnegative solution of $(P)$. Then, $u$ is positive solution.
\end{corollary}
\noindent{\bf{Proof:}}  If  $\Omega\subset \mathbb{R}^N$ is a bounded domain, the Lemma \ref{22.21} implies that $u\in C^{1}(\overline{\Omega})$.
Using this fact, in
the sequel, we fix  $M_1>\max\big\{\lVert \nabla u\lVert_{L^{\infty}(\overline{\Omega})}, 1\big\}$ and
\begin{align*}
\varphi(t)=\left\{\begin{array}{lr}
\;\;\; \;\phi(t)\;\;\;, \;\;\;\;\;\;\text{ for } 0< t\leq M_1\;\;\;\vspace*{0.2cm}\\
\dfrac{\phi(M_1)}{M_1^{\beta-2}}t^{\beta-2}\;,\text{ for } t\geq M_1
\end{array}\right.,
\end{align*}
where $\beta$ is given in the hypothesis $(\phi_{5})$. Still by condition $(\phi_5)$, there are $\alpha_1, \alpha_2>0$ satisfying
\begin{align}\label{2.23}
\varphi(|y|)|y|^2=	\phi(|y|)|y|^2\geq\alpha_1|y|^{\beta} \;\;\;\text{ and }\;\;\;|\varphi(|y|)y|\leq \alpha_2 |y|^{\beta-1},\;\;\;\forall y \in\mathbb{R}^N.
\end{align}
Now, consider the vector measurable functions $G:\Omega\times \mathbb{R} \times \mathbb{R}^N \longrightarrow \mathbb{R}^N$ given by $\linebreak G(x,t,p)=\frac{1}{\alpha_1}\varphi (|p|)p$. From \eqref{2.23},
\begin{align}
|G(x,t,p)|\leq\frac{\alpha_2}{\alpha_1}|p|^{\beta-1} \;\;\;\text{ and }\;\;\;pG(x,t,p)\geq |p|^{\beta-1} , \;\;\text{ for all }\; (x,t,p)\in \Omega\times \mathbb{R} \times \mathbb{R}^N.
\end{align}

We next will  consider the scalar measurable function $Z:\Omega\times \mathbb{R} \times \mathbb{R}^N \longrightarrow \mathbb{R}$ given by $$Z(x,t,p)=\frac{1}{\alpha_1}\big(V(x)\phi (|t|)t-\left(\int_{\mathbb {R}^{N}} \dfrac{K(y)F(u(y))}{|x-y|^\lambda}dy\right)K(x)f(t)\big).$$ By $(f_1)$, there will be a constant $C_1>0$ satisfying
\begin{align}
K(x)|f(t)|\leq C_1K(x)a
(|t|)|t|+ C_1\phi_*(|t|)|t|, \;\;\;\forall t\in\mathbb{R}^N\;\text{ and }\;x\in\mathbb{R}^N.
\end{align}

Fix $M\in(0,\infty)$. Through the condition $(\phi_{5})$ and by a simple computation yields there exists $C_2=C_2(M) > 0$ verifying
%
\begin{align*}
|Z(x,t,p)|\leq C_2|t|^{\beta-1},\; \text{ for every }\; (x,t,p)\in \Omega\times (-M,M) \times \mathbb{R}^N  . 
\end{align*}
By the arbitrariness of $M$, we can conclude that functions $G$ and $B$ fulfill the structure required by Trudinger \cite{Trudinger}. Also, as $u$ is a weak solution of $(P)$, we infer that $u$ is a quasilinear problem solution
\begin{align*}
-div\, G(x,u,\nabla u(x))+Z(x,u,\nabla u(x))=0\;\text{ in }\Omega.
\end{align*}
By [\citenum{Trudinger}, Theorem 1.1], we deduce that $u>0$ in $\Omega$. By the arbitrariness of $\Omega$ , we conclude that $u>0$ in $\mathbb{R}^{N}$.

\qed

\subsection{Proof of Theorem \ref{Teo2}} Assuming the assumptions of Theorem \ref{Teo2}, the arguments at the beginning of this Section 4, guarantees the existence of a nonnegative ground state solution for problem $(P)$. 

\qed

\section{Acknowledgments} 
The authors are grateful to the Paraíba State Research Foundation (FAPESQ), Brazil, and the Conselho Nacional de Desenvolvimento Científico e Tecnológico (CNPq), Brazil, whose funds partially supported this paper.

\end{document}